\documentclass[11pt,reqno]{amsart}
\usepackage[utf8]{inputenc}
\usepackage{amsfonts,amsthm,amsmath,amssymb,thmtools}
\usepackage{booktabs}
\usepackage{boldline,makecell}
\usepackage{array,diagbox}
\usepackage{graphicx}
\usepackage[dvipsnames,table]{xcolor}   % »table« optons loads »colortbl« package
\usepackage{enumerate}
\usepackage{hyperref}
\usepackage{color}
\usepackage{tikz-cd}
\usepackage{transparent}
\usepackage[margin=1in]{geometry}
\usepackage[
    maxbibnames=99,
    backend=biber,
    style=alphabetic,
    giveninits=true
]{biblatex}
\DeclareFieldFormat{pages}{#1}
\renewbibmacro{in:}{%
  \ifentrytype{article}
    {}
    {\bibstring{in}%
     \printunit{\intitlepunct}}}
\DeclareFieldFormat
[article,inbook,incollection,inproceedings,patent,thesis,unpublished]
  {title}{\mkbibemph{#1}}
\DeclareFieldFormat{journaltitle}{#1\isdot}
\DeclareFieldFormat[article]{volume}{\mkbibbold{#1}}
\DeclareFieldFormat[article]{number}{\bibstring{number}\addnbspace #1}

\renewbibmacro*{journal+issuetitle}{%
  \usebibmacro{journal}%
  \setunit*{\addspace}%
  \iffieldundef{series}
    {}
    {\newunit
     \printfield{series}%
     \setunit{\addspace}}%
  \printfield{volume}%
  \setunit{\addspace}%
  \usebibmacro{issue+date}%
  \setunit{\addcomma\space}%
  \printfield{number}%
  \setunit{\addcolon\space}%
  \usebibmacro{issue}%
  \setunit{\addcomma\space}%
  \printfield{eid}
  \newunit}

\newtheoremstyle{mystyle}%                % Name
  {}%                                     % Space above
  {}%                                     % Space below
  {\itshape}%                                     % Body font
  {}%                                     % Indent amount
  {\bfseries}%                            % Theorem head font
  {.}%                                    % Punctuation after theorem head
  { }%                                    % Space after theorem head, ' ', or \newline
  {\thmname{#1}\thmnumber{ #2}\thmnote{ (#3)}}%                                     % Theorem head spec (can be left empty, meaning `normal')

\theoremstyle{mystyle}
\newtheorem{Thm}{Theorem}[section]
\newtheorem{Lem}[Thm]{Lemma}
\newtheorem{Cor}[Thm]{Corollary}
\newtheorem{Prop}[Thm]{Proposition}
\newtheorem{Conj}[Thm]{Conjecture}

\newtheorem{Con}{Conjecture}

\theoremstyle{definition}

\theoremstyle{remark}
\newtheorem{Rmk}[Thm]{Remark}

\declaretheoremstyle[%
  spaceabove=3pt,%reduce or increase between theorem and proof
  spacebelow=10pt,%reduce or increase
  headfont=\normalfont\itshape,%
  postheadspace=.5em,%
  qed=\qedsymbol%
]{mystyle2} 
\declaretheorem[name={Proof},style=mystyle2,unnumbered,
]{pf}

\newcommand{\R}{\mathbb{R}}
\newcommand{\Z}{\mathbb{Z}}
\newcommand{\Q}{\mathbb{Q}}
\newcommand{\CP}{\mathbb{CP}}

\newcommand{\Rone}{\mathbb{Z}}
\newcommand{\Rmor}[1]{\mathbb{Z}^{#1}}
\newcommand{\Tone}[1]{\mathbb{Z}_{#1}}
\newcommand{\Tmor}[2]{(\mathbb{Z}_{#1})^{#2}}

\title{Lee Filtration Structure of Torus Links}

\author{Qiuyu Ren}
\address{Department of Mathematics, University of California, Berkeley, Berkeley, CA 94709, USA}
\email{qiuyu\_ren@berkeley.edu}
\begin{document}

\begin{abstract}
We determine the quantum filtration structure of the Lee homology of all torus links. In particular, this determines the $s$-invariant of a torus link equipped with any orientation. In the special case $T(n,n)$, our result confirms a conjecture of Pardon, as well as a conjecture of Manolescu-Marengon-Sarkar-Willis which establishes an adjunction-type inequality of the $s$-invariant for cobordisms in $k\overline{\CP^2}$. We also give a few applications of this adjunction inequality.
\end{abstract}

\maketitle

\section{Introduction}
Khovanov homology is a link invariant introduced by Khovanov \cite{khovanov2000categorification}, which categorifies the Jones polynomial. Despite its computability for every given link diagram, few families of links have had their Khovanov homology fully determined. Notable examples where the complete answer is known include alternating links \cite{lee2005endomorphism} (more generally, quasi-alternating links \cite{manolescu2008khovanov}), and the torus links $T(2,m)$ \cite{khovanov2000categorification}, $T(3,m)$ \cite{turner2008spectral,stovsic2009khovanov,gillam2012knot,benheddi2017khovanov}.

Specifically, in the case of positive torus links $T(n,m)$, although the Jones polynomials have a well-known closed form, this is far from true for Khovanov homology. The investigation of the Khovanov homology of torus links dates back at least to Sto\v{s}i\'c \cite{stovsic2007homological,stovsic2009khovanov}, where he calculated, for example, some low ($h\le4$) homological degree parts of $Kh(T(n,m))$ and the highest ($h=2kn^2$) homological degree part of $Kh(T(2n,2kn))$. More interestingly, he showed the existence of the ``stable Khovanov homology groups'' of $T(n,m)$ as $m\to\infty$, which has since been extensively investigated, see for example \cite{gorsky2013stable}. It is also worth remarking that the Khovanov-Rozansky triply graded homology of $T(n,m)$ was completely determined \cite{hogancamp2019torus}. However, a comprehensive understanding of the ordinary Khovanov homology for torus links remains elusive.

Nevertheless, many useful knot or link invariants that are more computable and possess desirable properties can be derived from Khovanov homology or its variants. A notable example is Rasmussen's $s$-invariant for knots, extracted from Lee homology \cite{lee2005endomorphism}, whose values on torus knots were computed and played a crucial role in providing the first gauge-theory-free proof of Milnor's conjecture on the slice genus of torus knots \cite{rasmussen2010khovanov}. In the case of links, the quantum filtration structure of the Lee homology can be considered as a natural generalization of the $s$-invariant for knots, encompassing the generalization proposed by Beliakova-Wehrli \cite{beliakova2008categorification} and Pardon \cite{pardon2012link} as special cases. In this paper, we completely determine the quantum filtration structure of the Lee homology of all torus links.

For ease of exposition, we only state our result in terms of Beliakova-Wehrli's $s$-invariant for oriented links and Pardon's invariants as the bigraded dimension of the associated graded vector space of the Lee homology. The actually quantum filtration structure, i.e. the quantum filtration degree function $q\colon Kh_{Lee}\to\Z\sqcup\{+\infty\}$, will become apparent during the proof of Theorem~\ref{thm:filtration}.

In the statements below, let $n,m$ be two positive integers, $d=\mathrm{gcd}(n,m)$, $n_1=n/d$, $m_1=m/d$. For $p,q\ge0$ with $p+q=d$, let $T(n,m)_{p,q}$ denote the torus link $T(n,m)$ equipped with an orientation in which $p$ of the components are oriented oppositely to the other $q$ components.

\begin{Thm}\label{thm:s}
The $s$-invariant of $T(n,m)_{p,q}$, over any coefficient field, is given by
$$s(T(n,m)_{p,q})=(n_1|p-q|-1)(m_1|p-q|-1)-2\min(p,q).$$ 
\end{Thm}

Here, the $s$-invariant of an oriented link over any field of characteristic not equal to $2$ is defined in the same way as in \cite{beliakova2008categorification}. Over characteristic $2$, one should use the Bar-Natan deformation (cf. \cite{bar2005khovanov}) of Khovanov homology instead of the Lee deformation.

By construction of the Khovanov/Lee complex, the Lee homology of $T(n,m)_{p,q}$, as a homologically graded and quantum filtered vector space, equals to that of the positive torus link $T(n,m):=T(n,m)_{d,0}$ up to a bidegree shift. Thus we will only state the structure of $Kh_{Lee}(T(n,m))$. By Lee \cite[Proposition~4.3]{lee2005endomorphism}, $Kh_{Lee}(T(n,m))$ as a homologically graded vector space is determined by the linking matrix of $T(n,m)$. Explicitly, $$\dim Kh_{Lee}^{2n_1m_1pq}(T(n,m))=\begin{cases}2\binom{d}{q},&p\ne q\\\binom{d}{q},&p=q,\end{cases}$$ for every pair of nonnegative integers $(p,q)$ with $p+q=d$, with other graded components being zero. The following theorem determines (the isomorphic type of) its quantum filtration structure.

\begin{Thm}\label{thm:filtration}
The associated graded vector space of the Lee homology (over $\Q$) of the positive torus link $T(n,m)$ is determined by
\begin{align*}
&\dim gr(Kh_{Lee}(T(n,m)))^{2n_1m_1pq,6n_1m_1pq+s(T(n,m)_{p,q})+2r-1}\\
=&\,\begin{cases}\dim(d-r,r),&r=0,\ p\ne q\\\dim(d-r,r)+\dim(d-r+1,r-1),&0<r<\min(p,q)+1,\ p\ne q\\\dim(d-r+1,r-1),&r=\min(p,q)+1,\ p\ne q\\\dim(d-r,r),&0\le r\le\min(p,q),\ p=q,\end{cases}
\end{align*}
for every pair of nonnegative integers $(p,q)$ with $p+q=d$, with all other bigraded components being zero. Here $(a,b)$ denotes the irreducible representation of the symmetric group $S_{a+b}$ given by the two-row Young diagram $(a,b)$. Thus $$\dim(d-r,r)=\binom{d}{r}-\binom{d}{r-1},\ 0\le r\le\frac{d}{2}.$$
\end{Thm}

The appearance of $S_d$-representations in Theorem~\ref{thm:filtration} is no coincidence. Indeed, we will show that $Kh_{Lee}(T(n,m))$ carries a filtered $S_d$-action, and determine its structure as a filtered $S_d$-representation.

In the special case $m=n$, Theorem~\ref{thm:filtration} confirms a conjecture of Pardon \cite[Section~5.2]{pardon2012link}.

Since taking the mirror image of a link has the effect of taking the dual on the Lee homology, the quantum filtration structure of $T(n,-m)$ is determined by that of $T(n,m)$. We can also read off its $s$-invariants as follows (cf. the second paragraph in the proof of Theorem~\ref{thm:s}).
\begin{Cor}\label{cor:bar_s}
The $s$-invariant, over $\Q$, of $T(n,-m)_{p,q}$ is given by $$s(T(n,-m)_{p,q})=\begin{cases}-(n_1|p-q|-1)(m_1|p-q|-1),&p\ne q\\1,&p=q.\end{cases}$$
\par\vspace{-2\baselineskip}\qed
\end{Cor}

As observed by Manolescu-Marengon-Sarkar-Willis \cite{manolescu2023generalization}, Theorem~\ref{thm:s} in the special case $m=n$ implies the following corollary, which is an adjunction-type inequality for $s$-invariants of null-homologous oriented links in a connected sum of $(S^1\times S^2)$'s, as defined in their paper. We remark that, however, one cannot deduce an adjunction-type inequality from Corollary~\ref{cor:bar_s} using the same proof.

\begin{Cor}\label{cor:adjunction1}
If $\Sigma\subset Z=(I\times\ell(S^1\times S^2))\#k\overline{\CP^2}$ is an oriented cobordism between two null-homologous oriented links $L_0,L_1\subset\ell(S^1\times S^2)$ with $\pi_0(L_1)\to\pi_0(\Sigma)$ surjective, then
\begin{equation}\label{eq:adjunction1}
s(L_1)\le s(L_0)-\chi(\Sigma)-[\Sigma]^2-|[\Sigma]|'.\nonumber
\end{equation}
Here $|[\Sigma]|'$ is defined as the sum $\sum_{i=1}^k|[\Sigma]\cdot z_i|$ where $z_1,\cdots,z_k\in H_2(Z)\cong\Z^{\ell+k}$ are the generators coming from the $\overline{\CP^2}$ factors.
\end{Cor}

Corollary~\ref{cor:adjunction1} holds over any coefficient field as long as $\ell=0$. For $\ell>0$, in \cite{manolescu2023generalization}, the $s$-invariant is only defined over fields with characteristic not equal to $2$, although we expect everything to hold in characteristic $2$ as well.

The term $[\Sigma]^2$ above is well-defined, and the term $|[\Sigma]|'$ is independent of the choice of the decomposition $Z=(I\times\ell(S^1\times S^2))\#k\overline{\CP^2}$, both thanks to the links $L_i$ being null-homologous. Of course, one may also dualize and obtain a similar adjunction inequality in $(I\times\ell(S^1\times S^2))\#k\CP^2$ (cf. Section~\ref{sec:s_adj}).

In practice, the special case $\ell=0$ might be the most useful, where $s$ reduces to Beliakova-Wehrli's generalization of classical Rasmussen's $s$-invariant (defined in Section~\ref{sec:s_pf}). In particular, this opens a new approach to detect exotic $k\overline{\CP^2}$ whose existence is not yet known, for example by modifying the constructions in \cite{manolescu2021zero}. When $\ell=0$, $L_0=\emptyset$, and $L_1=K$ is a knot, Corollary~\ref{cor:adjunction1} takes the following form.

\begin{Cor}[{\cite[Conjecture~9.8]{manolescu2023generalization}}]\label{cor:adjunction}
If $(\Sigma,K)\subset((k\overline{\CP^2})\backslash B^4,S^3)$ is a connected orientable properly embedded surface with boundary a knot $K$, then
$$s(K)\le1-\chi(\Sigma)-[\Sigma]^2-|[\Sigma]|.$$
Here $|[\Sigma]|$ is the $L^1$-norm of $[\Sigma]$ (denoted as $|[\Sigma]|'$ in Corollary~\ref{cor:adjunction1}).\qed
\end{Cor}

The adjunction-type inequality parallel to Corollary~\ref{cor:adjunction} for the $\tau$-invariant in knot Floer homology was established by Ozsv\'ath and Szab\'o two decades ago \cite{ozsvath2003knot}. The inequality parallel to the more general Corollary~\ref{cor:adjunction1} appeared recently (for knots, and in some special cases for links) in the work of Hedden and Raoux \cite{hedden2023knot}. In fact, their inequalities were stated for more general smooth $4$-manifolds. This naturally leads us to the question of whether Corollary~\ref{cor:adjunction1} and Corollary~\ref{cor:adjunction} can be generalized to those settings. Of course, this (in its full generality) entails generalizing the $s$-invariant to rationally null-homologous links in arbitrary closed oriented $3$-manifolds. It is worth remarking, however, that in order to successfully construct exotic $k\overline{\CP^2}$'s detectable by the $s$-invariant using a modified version of the construction in \cite{manolescu2021zero}, one should hope that Corollary~\ref{cor:adjunction} does not have a generalization applicable to arbitrary simply connected negative definite $4$-manifolds.

Two applications of the adjunction inequality for the $s$-invariant will be given in Section~\ref{sec:adjunction}.\bigskip

We summarize the paper structure and briefly describe the proofs of the main results as follows.

In Section~\ref{sec:s_bound}, we state a ``graphical lower bound'', Theorem~\ref{thm:lower_bound} (see also Figure~\ref{fig:T66} and Figure~\ref{fig:T76}), for the Khovanov homology of the torus links $T(n,n)$ and the torus knots $T(n+1,n)$. In the case of $T(n,n)$, in homological degrees $2pq$ with $p+q=n$, the bound is sharp, and the nonvanishing groups with the lowest quantum degrees are $\Z$, which also give rises to Lee homology generators. As we shall see in Section~\ref{sec:s_pf}, this implies Theorem~\ref{thm:s} in the special case $m=n$. The general case is then proved in Section~\ref{sec:s_adj} inductively using the adjunction inequality (Corollary~\ref{cor:adjunction1}). The proof of Theorem~\ref{thm:lower_bound}, which follows the induction scheme set up by Sto\v{s}i\'c \cite{stovsic2007homological,stovsic2009khovanov}, is a cumbersome verification that is not illuminating and is deferred to Section~\ref{sec:proof}. 

As further illustrations of the power of the adjunction inequality, we give two applications. In Section~\ref{sec:s_sb}, we provide an optimal bound on the change of the $s$-invariant when full twists are applied to an oriented link. In particular, the $s$-invariant grows linearly when the twist number is sufficiently large, answering positively a question in \cite{manolescu2023generalization}. In Section~\ref{sec:s_genus}, we show there exist knots in $S^3$ with simultaneously large $\CP^2$-genus and $\overline{\CP^2}$-genus.

Section~\ref{sec:filtration} is independent of Section~\ref{sec:adjunction}. We exploit an $S_d$-symmetry and decompose $Kh_{Lee}(T(n,m))$ into irreducible $S_d$-representations. By working equivariantly, we are able deduce Theorem~\ref{thm:filtration} inductively from Theorem~\ref{thm:s}.

Finally, in Section~\ref{sec:questions}, we state a numerical observation as an open question. In particular, we propose a conjectural (recursive) formula for the rational Khovanov homology of $T(n,n)$.

\subsection*{Acknowledgement}
I would like to thank my advisor Ian Agol for insights, discussions, and encouragement. I thank Qianhe Qin for introducing me to this question, Marco Marengon for kindly providing me with the proof of Proposition~\ref{prop:CP2_genus}, Michael Willis for feedback on a first draft of this paper, and Ciprian Manolescu and Melissa Zhang for many helpful discussions. I also thank the referee for a careful read of the paper and many valuable suggestions.

\section{The \texorpdfstring{$s$}{s}-invariants of \texorpdfstring{$T(n,m)$}{T(n,m)}}\label{sec:s}
\subsection{A graphical lower bound for \texorpdfstring{$Kh(T(n,n))$}{Kh(T(n,n))} and \texorpdfstring{$Kh(T(n+1,n))$}{Kh(T(n+1,n))}}\label{sec:s_bound}
In this section we state the core technical result of this paper, which gives a ``graphical lower bound'' on the Khovanov homology of $T(n,n)$ and $T(n+1,n)$, whose proof is postponed to Section~\ref{sec:proof}. We assume the reader is familiar with Khovanov homology and Lee homology, and refer to the literature mentioned in the introduction if otherwise. See also \cite{bar2002khovanov} for a short introduction.

Before stating the theorem, we introduce two families of auxiliary functions $q_{n,n},q_{n+1,n}\colon\Z\to\Z\cup\{+\infty\}$, which serve as quantum lower bounds for the Khovanov homology of $T(n,n),T(n+1,n)$ in various homological degrees, respectively. Let $$h_{max}(T(n,n)):=\lfloor n^2/2\rfloor,\ h_{max}(T(n+1,n)):=\lfloor n^2/2\rfloor+\lfloor n/2\rfloor.$$

The functions $q_{n,n}$ are defined by:
\begin{enumerate}[(a)]
\item $q_{n,n}(h)=+\infty$ for $h<0$ or $h>h_{max}(T(n,n))$;
\item $q_{n,n}(0)=n^2-2n$;
\item If $p+q=n$, $p\ge q>0$, then for $2(p+1)(q-1)<h\le 2pq$, $$q_{n,n}(h)=n^2+2\left\lceil\frac h2\right\rceil-2p.$$
\end{enumerate}

\begin{figure}[t]
\centering
\resizebox{\textwidth}{!}{
\begin{tabular}{|c||c|c|c|c|c|c|c|c|c|c|c|c|c|c|c|c|c|c|c|}
\noalign{\global\arrayrulewidth=0.4pt}\hline
\backslashbox{\!$q$\!}{\!$h$\!} & $0$ & $1$ & $2$ & $3$ & $4$ & $5$ & $6$ & $7$ & $8$ & $9$ & $10$ & $11$ & $12$ & $13$ & $14$ & $15$ & $16$ & $17$ & $18$ \\
\hline
\hline
$54$  &   &   &   &   &   &   &   &   &   &   &   &   &   &   &   &   &   &   & $ \Rmor{5} $ \\
\hline
$52$  &   &   &   &   &   &   &   &   &   &   &   &   &   &   &   &   &   &   & $ \Rmor{9} $ \\
\hline
$50$  &   &   &   &   &   &   &   &   &   &   &   &   &   &   &   & $ \Rmor{5} $ & $ \Rmor{9} $ &   & $ \Rmor{5} $ \\
\hline
$48$  &   &   &   &   &   &   &   &   &   &   &   &   &   &   & $ \Tone{3} $ & $ \Rmor{6} $ & $ \Rmor{14} $ &   & $ \cellcolor{green}\Rone $ \\
\cline{1-18}\arrayrulecolor{red}\noalign{\global\arrayrulewidth=1.5pt}\cline{19-20}\arrayrulecolor{black}\noalign{\global\arrayrulewidth=0.4pt}
$46$  &   &   &   &   &   &   &   &   &   &   &   &   & $ \Rone $ & $ \Rmor{6} $ & $ \Rmor{5} $ & $ \Rone $ & $ \Rmor{6} $ &   &   \\
\hline
$44$  &   &   &   &   &   &   &   &   &   &   &   &   & $ \Tone{2} \oplus \Tone{5} $ & $ \Rmor{2} \oplus \Tmor{2}{5} $ & $ \Rmor{6} $ &   & $ \cellcolor{green}\Rone $ &   &   \\
\cline{1-16}\arrayrulecolor{red}\noalign{\global\arrayrulewidth=1.5pt}\cline{17-18}\arrayrulecolor{black}\noalign{\global\arrayrulewidth=0.4pt}\cline{19-20}
$42$  &   &   &   &   &   &   &   &   &   &   & $ \Tone{2} $ & $ \Rmor{2} $ & $ \Rmor{6} $ & $ \Tone{2} $ & $ \Rone $ &   &   &   &   \\
\cline{1-14}\arrayrulecolor{red}\noalign{\global\arrayrulewidth=1.5pt}\cline{15-16}\arrayrulecolor{black}\noalign{\global\arrayrulewidth=0.4pt}\cline{17-20}
$40$  &   &   &   &   &   &   &   &   &   & $ \Rmor{2} $ & $ \Rmor{5} \oplus \Tone{2} $ & $ \Tone{2} \oplus \Tone{5} $ & $ \Rmor{2} $ &   &   &   &   &   &   \\
\cline{1-12}\arrayrulecolor{red}\noalign{\global\arrayrulewidth=1.5pt}\cline{13-14}\arrayrulecolor{black}\noalign{\global\arrayrulewidth=0.4pt}\cline{15-20}
$38$  &   &   &   &   &   &   &   & $ \Rone $ &   & $ \Rone \oplus \Tone{2} $ & $ \Rmor{7} $ &   &   &   &   &   &   &   &   \\
\hline
$36$  &   &   &   &   &   & $ \Rone $ &   & $ \Rone \oplus \Tone{2} $ & $ \Rmor{2} $ &   & $ \cellcolor{green}\Rone $ &   &   &   &   &   &   &   &   \\
\cline{1-10}\arrayrulecolor{red}\noalign{\global\arrayrulewidth=1.5pt}\cline{11-12}\arrayrulecolor{black}\noalign{\global\arrayrulewidth=0.4pt}\cline{13-20}
$34$  &   &   &   &   &   & $ \Rone $ & $ \Rone $ & $ \Tone{2} $ & $ \Rone $ &   &   &   &   &   &   &   &   &   &   \\
\cline{1-8}\arrayrulecolor{red}\noalign{\global\arrayrulewidth=1.5pt}\cline{9-10}\arrayrulecolor{black}\noalign{\global\arrayrulewidth=0.4pt}\cline{11-20}
$32$  &   &   &   & $ \Rone $ & $ \Rone $ &   & $ \Rone $ &   &   &   &   &   &   &   &   &   &   &   &   \\
\cline{1-6}\arrayrulecolor{red}\noalign{\global\arrayrulewidth=1.5pt}\cline{7-8}\arrayrulecolor{black}\noalign{\global\arrayrulewidth=0.4pt}\cline{9-20}
$30$  &   &   &   & $ \Tone{2} $ & $ \Rone $ &   &   &   &   &   &   &   &   &   &   &   &   &   &   \\
\cline{1-4}\arrayrulecolor{red}\noalign{\global\arrayrulewidth=1.5pt}\cline{5-6}\arrayrulecolor{black}\noalign{\global\arrayrulewidth=0.4pt}\cline{7-20}
$28$  &   &   & $ \Rone $ &   &   &   &   &   &   &   &   &   &   &   &   &   &   &   &   \\
\cline{1-2}\arrayrulecolor{red}\noalign{\global\arrayrulewidth=1.5pt}\cline{3-4}\arrayrulecolor{black}\noalign{\global\arrayrulewidth=0.4pt}\cline{5-20}
$26$  & $ \Rone $ &   &   &   &   &   &   &   &   &   &   &   &   &   &   &   &   &   &   \\
\hline
$24$  & $ \cellcolor{green}\Rone $ &   &   &   &   &   &   &   &   &   &   &   &   &   &   &   &   &   &   \\
\cline{1-1}\arrayrulecolor{red}\noalign{\global\arrayrulewidth=1.5pt}\cline{2-2}\arrayrulecolor{black}\noalign{\global\arrayrulewidth=0.4pt}\cline{3-20}
\end{tabular}
}
\caption{The Khovanov homology of $T(6,6)$. The lower bound given by $q_{6,6}$ is drawn red. The homology groups $Kh^{2pq,q_{6,6}(2pq)}$ are colored green.}
\label{fig:T66}
\end{figure}

The functions $q_{n+1,n}$ are defined by:
\begin{enumerate}[(a)]
\item $q_{n+1,n}(h)=+\infty$ for $h<0$ or $h>h_{max}(T(n+1,n))$;
\item If $p+q=n$, $p\ge q>0$, then $$q_{n+1,n}(2pq+1)=q_{n,n}(2pq+1)+n-3;$$
\item For other $0\le h\le h_{max}(T(n,n))$, $$q_{n+1,n}(h)=q_{n,n}(h)+n-1;$$
\item For $h_{max}(T(n,n))\le h\le h_{max}(T(n+1,n))$, $$q_{n+1,n}(h)=\left\lfloor\frac{n^2}2\right\rfloor+2h-1.$$
\end{enumerate}
Note the two definitions of $q_{n+1,n}(h_{max}(T(n,n)))$ via (c)(d) above agree.

\begin{figure}[t]
\centering
\resizebox{\textwidth}{!}{
\begin{tabular}{|c||c|c|c|c|c|c|c|c|c|c|c|c|c|c|c|c|c|c|c|c|c|c|}
\noalign{\global\arrayrulewidth=0.4pt}\hline
\backslashbox{\!$q$\!}{\!$h$\!} & $0$ & $1$ & $2$ & $3$ & $4$ & $5$ & $6$ & $7$ & $8$ & $9$ & $10$ & $11$ & $12$ & $13$ & $14$ & $15$ & $16$ & $17$ & $18$ & $19$ & $20$ & $21$ \\
\hline
\hline
$61$  &   &   &   &   &   &   &   &   &   &   &   &   &   &   &   &   &   &   &   &   & $ \Tone{3} $ & $ \Tone{2} $ \\
\hline
$59$  &   &   &   &   &   &   &   &   &   &   &   &   &   &   &   &   &   &   &   &   & $ \Tone{2} \oplus \Tone{3} $ & $ \Tone{2} $ \\
\cline{1-22}\arrayrulecolor{red}\noalign{\global\arrayrulewidth=1.5pt}\cline{23-23}\arrayrulecolor{black}\noalign{\global\arrayrulewidth=0.4pt}
$57$  &   &   &   &   &   &   &   &   &   &   &   &   &   &   &   &   & $ \Rone $ &   & $ \Tone{4} $ & $ \Rone $ & $ \Tone{2} $ &   \\
\cline{1-21}\arrayrulecolor{red}\noalign{\global\arrayrulewidth=1.5pt}\cline{22-22}\arrayrulecolor{black}\noalign{\global\arrayrulewidth=0.4pt}\cline{23-23}
$55$  &   &   &   &   &   &   &   &   &   &   &   &   &   &   &   &   & $ \Rone \oplus \Tone{2} $ & $ \Rone $ & $ \Tmor{2}{2} $ & $ \Tone{2} \oplus \Tone{3} $ &   &   \\
\cline{1-20}\arrayrulecolor{red}\noalign{\global\arrayrulewidth=1.5pt}\cline{21-21}\arrayrulecolor{black}\noalign{\global\arrayrulewidth=0.4pt}\cline{22-23}
$53$  &   &   &   &   &   &   &   &   &   &   &   &   &   &   & $ \Rone $ & $ \Rmor{2} $ & $ \Tone{2} $ & $ \Rone \oplus \Tone{2} $ & $ \Rone \oplus \Tone{2} $ &   &   &   \\
\cline{1-19}\arrayrulecolor{red}\noalign{\global\arrayrulewidth=1.5pt}\cline{20-20}\arrayrulecolor{black}\noalign{\global\arrayrulewidth=0.4pt}\cline{21-23}
$51$  &   &   &   &   &   &   &   &   &   &   &   &   & $ \Rone $ & $ \Rone $ & $ \Tmor{2}{2} $ & $ \Rmor{2} \oplus \Tone{2} $ & $ \Rone $ & $ \Tone{2} $ &   &   &   &   \\
\cline{1-18}\arrayrulecolor{red}\noalign{\global\arrayrulewidth=1.5pt}\cline{19-19}\arrayrulecolor{black}\noalign{\global\arrayrulewidth=0.4pt}\cline{20-23}
$49$  &   &   &   &   &   &   &   &   &   &   &   &   & $ \Tone{2} \oplus \Tone{5} $ & $ \Rmor{3} \oplus \Tone{2} $ & $ \Rone \oplus \Tone{2} $ & $ \Tone{2} $ & $ \Rone $ &   &   &   &   &   \\
\cline{1-16}\arrayrulecolor{red}\noalign{\global\arrayrulewidth=1.5pt}\cline{17-18}\arrayrulecolor{black}\noalign{\global\arrayrulewidth=0.4pt}\cline{19-23}
$47$  &   &   &   &   &   &   &   &   &   &   & $ \Tone{2} $ & $ \Rmor{3} $ & $ \Rone $ & $ \Tmor{2}{2} $ & $ \Rone $ &   &   &   &   &   &   &   \\
\cline{1-14}\arrayrulecolor{red}\noalign{\global\arrayrulewidth=1.5pt}\cline{15-16}\arrayrulecolor{black}\noalign{\global\arrayrulewidth=0.4pt}\cline{17-23}
$45$  &   &   &   &   &   &   &   &   &   & $ \Rmor{2} $ & $ \Tone{2} $ & $ \Rone \oplus \Tmor{2}{2} $ & $ \Rmor{2} $ &   &   &   &   &   &   &   &   &   \\
\cline{1-13}\arrayrulecolor{red}\noalign{\global\arrayrulewidth=1.5pt}\cline{14-14}\arrayrulecolor{black}\noalign{\global\arrayrulewidth=0.4pt}\cline{15-23}
$43$  &   &   &   &   &   &   &   & $ \Rone $ &   & $ \Rone \oplus \Tone{2} $ & $ \Rmor{2} $ & $ \cellcolor{yellow}\Tone{2} $ &   &   &   &   &   &   &   &   &   &   \\
\cline{1-12}\arrayrulecolor{red}\noalign{\global\arrayrulewidth=1.5pt}\cline{13-13}\arrayrulecolor{black}\noalign{\global\arrayrulewidth=0.4pt}\cline{14-23}
$41$  &   &   &   &   &   & $ \Rone $ &   & $ \Rone \oplus \Tone{2} $ & $ \Rmor{2} $ &   & $ \Rone $ &   &   &   &   &   &   &   &   &   &   &   \\
\cline{1-10}\arrayrulecolor{red}\noalign{\global\arrayrulewidth=1.5pt}\cline{11-12}\arrayrulecolor{black}\noalign{\global\arrayrulewidth=0.4pt}\cline{13-23}
$39$  &   &   &   &   &   & $ \Rone $ & $ \Rone $ & $ \Tone{2} $ & $ \Rone $ &   &   &   &   &   &   &   &   &   &   &   &   &   \\
\cline{1-8}\arrayrulecolor{red}\noalign{\global\arrayrulewidth=1.5pt}\cline{9-10}\arrayrulecolor{black}\noalign{\global\arrayrulewidth=0.4pt}\cline{11-23}
$37$  &   &   &   & $ \Rone $ & $ \Rone $ &   & $ \Rone $ &   &   &   &   &   &   &   &   &   &   &   &   &   &   &   \\
\cline{1-6}\arrayrulecolor{red}\noalign{\global\arrayrulewidth=1.5pt}\cline{7-8}\arrayrulecolor{black}\noalign{\global\arrayrulewidth=0.4pt}\cline{9-23}
$35$  &   &   &   & $ \Tone{2} $ & $ \Rone $ &   &   &   &   &   &   &   &   &   &   &   &   &   &   &   &   &   \\
\cline{1-4}\arrayrulecolor{red}\noalign{\global\arrayrulewidth=1.5pt}\cline{5-6}\arrayrulecolor{black}\noalign{\global\arrayrulewidth=0.4pt}\cline{7-23}
$33$  &   &   & $ \Rone $ &   &   &   &   &   &   &   &   &   &   &   &   &   &   &   &   &   &   &   \\
\cline{1-2}\arrayrulecolor{red}\noalign{\global\arrayrulewidth=1.5pt}\cline{3-4}\arrayrulecolor{black}\noalign{\global\arrayrulewidth=0.4pt}\cline{5-23}
$31$  & $ \Rone $ &   &   &   &   &   &   &   &   &   &   &   &   &   &   &   &   &   &   &   &   &   \\
\cline{1-23}
$29$  & $ \Rone $ &   &   &   &   &   &   &   &   &   &   &   &   &   &   &   &   &   &   &   &   &   \\
\cline{1-1}\arrayrulecolor{red}\noalign{\global\arrayrulewidth=1.5pt}\cline{2-2}\arrayrulecolor{black}\noalign{\global\arrayrulewidth=0.4pt}\cline{3-23}
\end{tabular}
}
\caption{The Khovanov homology of $T(7,6)$. The lower bound given by $q_{7,6}$ is drawn red. The homology group $Kh^{2\cdot6-1,q_{7,6}(2\cdot6-1)}$ is colored yellow.}
\label{fig:T76}
\end{figure}

We now state the main technical theorem. See Figure~\ref{fig:T66} and Figure~\ref{fig:T76} for an illustration of the case $n=6$. The reader is warned that the same letter $q$ is (unfortunately) used for two different purposes: the quantum degree and an integer between $0$ to $n$ that is complementary to $p$. It should be clear from context which of these is referred to.

\begin{Thm}\label{thm:lower_bound}
\begin{enumerate}[(1)]
\item $Kh^{h,q}(T(n,n))=0$ for $q<q_{n,n}(h)$. Moreover, for every $p+q=n$, $p,q>0$, the saddle cobordism $T(n-2,n-2)\sqcup U\to T(n,n)$ induces an isomorphism $$Kh^{2(p-1)(q-1),q_{n-2,n-2}(2(p-1)(q-1))-1}(T(n-2,n-2)\sqcup U)\xrightarrow{\cong}Kh^{2pq,q_{n,n}(2pq)}(T(n,n)).$$
\item $Kh^{h,q}(T(n+1,n))=0$ for $q<q_{n+1,n}(h)$. Moreover, $Kh^{2n-1,q_{n+1,n}(2n-1)}(T(n+1,n))$ is torsion.
\end{enumerate}
\end{Thm}

% \begin{Thm}\label{thm:lower_bound}\mbox{}\hfill\vspace{-6pt}
% \begin{enumerate}[(1)]
% \item $Kh^{h,q}(T(n,n))=0$ for $q<q_{n,n}(h)$, where $q_{n,n}\colon\Z\to\Z\sqcup\{+\infty\}$ is the function defined by
% \begin{enumerate}[(a)]
% \item $q_{n,n}(0)=n^2-2n$;
% \item $q_{n,n}(h+2)=q_{n,n}(h+1)=q_{n,n}(h)+2$ for even $0\le h\le h_{max}(T(n,n))-2$ with $h\ne2pq$ for any $p+q=n$, where $h_{max}(T(n,n))=\lfloor\frac{n^2}{2}\rfloor$;
% \item $q_{n,n}(h+2)=q_{n,n}(h+1)=q_{n,n}(h)+4$ for $h=2pq$ for some $p+q=n$.
% \item $q_{n,n}(h)=+\infty$ for $h<0$ or $h>h_{max}(T(n,n))$.
% \end{enumerate}
% Moreover, for every $p+q=n$, $p,q>0$, the saddle cobordism $T(n-2,n-2)\sqcup U\to T(n,n)$ induces an isomorphism $$Kh^{2(p-1)(q-1),q_{n-2,n-2}(2(p-1)(q-1))-1}(T(n-2,n-2)\sqcup U)\xrightarrow{\cong}Kh^{2pq,q_{n,n}(2pq)}(T(n,n)).$$
% \item $Kh^{h,q}(T(n+1,n))=0$ for $q<q_{n+1,n}(h)$, where $q_{n+1,n}\colon\Z\to\Z\sqcup\{+\infty\}$ is the function defined by
% \begin{enumerate}[(a)]
% \item $q_{n+1,n}(h)=q_{n,n}(h)+n-1$ for $0\le h\le h_{max}(T(n,n))$ with $h\ne2pq+1$ for any $p+q=n$, $p,q>0$;
% \item $q_{n+1,n}(h)=q_{n,n}(h)+n-3$ for $0\le h\le h_{max}(T(n,n))$ with $h=2pq+1$ for some $p+q=n$, $p,q>0$;
% \item $q_{n+1,n}(h+1)=q_{n+1,n}(h)+2$ for $h_{max}(T(n,n))\le h<h_{max}(T(n+1,n))$, where $h_{max}(T(n+1,n))=h_{max}(T(n,n))+\lfloor\frac{n}{2}\rfloor$;
% \item $q_{n+1,n}(h)=+\infty$ for $h<0$ or $h>h_{max}(T(n+1,n))$.
% \end{enumerate}
% Moreover, $Kh^{2n-1,q_{n+1,n}(2n-1)}(T(n+1,n))$ is torsion.
% \end{enumerate}
% \end{Thm}

\begin{Cor}\label{cor:h=2pq}
For any $p+q=n$, $p,q\ge0$, we have $Kh^{2pq,q_{n,n}(2pq)}(T(n,n))=\Z$.
\end{Cor}
\begin{pf}
For $pq=0$, this follows from Sto\v{s}i\'c's calculation \cite[Theorem~3.4]{stovsic2007homological} that $Kh^{0,*}(T(n,n))=\Z$ for $*=(n-1)^2\pm1$ and $0$ otherwise. The general case then follows by induction using the last statement in Theorem~\ref{thm:lower_bound}(1).
\end{pf}

The case $p=n-1$, $q=1$ of Corollary~\ref{cor:h=2pq} confirms a conjecture of Sto\v{s}i\'c \cite[Conjecture~3.8]{stovsic2007homological}.

\subsection{Proof of Theorem~\ref{thm:s} for \texorpdfstring{$m=n$}{m=n}}\label{sec:s_pf}
We prove Theorem~\ref{thm:s} in the special case $m=n$, assuming Theorem~\ref{thm:lower_bound}. By symmetry we may assume $p\ge q$. The statement then takes the form $s(T(n,n)_{p,q})=(p-q-1)^2-2q$.

We induct on $n$. The base cases $n=1,2$ are easily checked. Assume $n\ge3$. If $q=0$, the result follows either from the sharpness of Kawamura-Lobb's inequality on $s$-invariants for nonsplit positive links \cite[Corollary~2.4]{abe2017characterization}, or Sto\v{s}i\'c's calculation of $Kh^{0,*}(T(n,n))$ mentioned above. Assume from now on $q>0$.

The linking matrix of $T(n,n)$ is $(\ell_{ij})_{1\le i,j\le n}$ where $\ell_{ij}=0$ if $i=j$ and $1$ otherwise. Thus by Lee~\cite[Proposition~4.3]{lee2005endomorphism}, $Kh_{Lee}^{2pq}(T(n,n))$ is the vector space spanned by the canonical generators $[s_\mathfrak{o}]$, one for each orientation $\mathfrak{o}$ of $T(n,n)$ realizing $T(n,n)_{p,q}$ (see also \cite{rasmussen2010khovanov}). Moreover, the Lee homology of $T(n,n)$ is that of $T(n,n)_{p,q}$ with a bidegree shift $[2pq]\{6pq\}$ \cite[Proposition~28, typo]{khovanov2000categorification}.

According to Beliakova-Wehrli's definition of the $s$-invariant \cite[Definition~7.1]{beliakova2008categorification}, $s(T(n,n)_{p,q})$ over $\Q$ equals the quantum filtration degree of $[s_\mathfrak{o}]\in Kh_{Lee}(T(n,n)_{p,q})$ plus $1$, where $\mathfrak{o}$ is any orientation that realizes $T(n,n)_{p,q}$. Taking into account the degree shift $6pq$, we need to prove $[s_\mathfrak{o}]\in Kh_{Lee}(T(n,n))$ has quantum filtration degree $(p-q-1)^2-2q-1+6pq=q_{n,n}(2pq)$.

Every element of $Kh_{Lee}^{2pq}(T(n,n))$ is a linear combination of these $[s_\mathfrak{o}]$'s, thus has quantum filtration degree no less than that of these $[s_\mathfrak{o}]$'s. Thus it remains to prove the lowest filtration level of $Kh_{Lee}^{2pq}(T(n,n))$ is $q_{n,n}(2pq)$.

We prove this by showing the equivalent statement that the homology group $Kh^{2pq,q_{n,n}(2pq)}(T(n,n))\otimes\Q\cong\Q$ (cf. Corollary~\ref{cor:h=2pq}) survives to the $E_\infty$ page in the Lee spectral sequence from $E_1=Kh(T(n,n))\otimes\Q$ to $Kh_{Lee}(T(n,n))$. The differential on the $E_r$ page has bidegree $(1,4r)$. By Theorem~\ref{thm:lower_bound}, we have $Kh^{2pq-1,*}(T(n,n))=0$ for $*<q_{n,n}(2pq-1)=q_{n,n}(2pq)$, so $Kh(T(n,n))\otimes\Q$ cannot be annihilated by differentials mapping into it. To see that all differentials out of it are zero, we observe that the naturality of the Lee spectral sequence applied to the cobordism $T(n-2,n-2)\to T(n,n)$ gives the following commutative diagram on page $E_r$:
$$\begin{tikzcd}
Kh^{2(p-1)(q-1),q_{n-2,n-2}(2(p-1)(q-1))}(T(n-2,n-2))\otimes\Q\ar[d,"d_r=0"]\ar[r,"\cong"]&Kh^{2pq,q_{n,n}(2pq)}(T(n,n))\otimes\Q\ar[d,"d_r"]\\
E_r^{2(p-1)(q-1)+1,q_{n-2,n-2}(2(p-1)(q-1))+4r}(T(n-2,n-2))\ar[r]&E_r^{2pq+1,q_{n,n}(2pq)+4r}(T(n,n)).
\end{tikzcd}$$
Here the vertical map on the left is zero by induction hypothesis and the horizontal map on the top is an isomorphism by Theorem~\ref{thm:lower_bound}(1). Consequently, the vertical map on the right is also zero.

The proof above works with $\Q$ replaced by any coefficient field with characteristic not equal to $2$. For characteristic $2$, one should replace the Lee homology with Bar-Natan homology, and Lee spectral sequences with Bar-Natan/Turner spectral sequences \cite{bar2005khovanov,turner2006calculating} (whose $r^{\text{th}}$ differential has bidegree $(1,2r)$), but everything else goes through. Thus the theorem is proved in the special case $m=n$.\qed

\subsection{Adjunction inequality and proof of Theorem~\ref{thm:s}}\label{sec:s_adj}
Manolescu-Marengon-Sarkar-Willis \cite[Theorem~6.10]{manolescu2023generalization} proved the adjunction inequality (Corollary~\ref{cor:adjunction1}) in the special case of null-homologous cobordisms, using the calculation of $s(T(2p,2p)_{p,p})$ in their paper. Having calculated all $s(T(n,n)_{p,q})$, Corollary~\ref{cor:adjunction1} in its full generality is proved exactly the same way. For completeness, we sketch their proof here. Then we apply this inequality to prove Theorem~\ref{thm:s} in its full generality.
\begin{proof}[Proof of Corollary~\ref{cor:adjunction1}]
Turn the cobordism upside down and reverse the ambient orientation, we obtain a cobordism $\Sigma^t$ in $\overline{Z^t}=(I\times\ell(S^1\times S^2))\#k\CP^2$ from $L_1$ to $L_0$ with $\pi_0(L_1)\to\pi_0(\Sigma^t)$ surjective. Choose embedded $2$-spheres $S_1,\cdots,S_k$ representing generators $\overline{z_1},\cdots,\overline{z_k}\in H_2(\overline{Z^t})$ coming from the $\CP^2$ factors. We may assume $\Sigma^t$ intersects each $S_i$ transversely, in some $p_i$ points positively and $q_i$ points negatively. Since $S_i$ has self-intersection $1$, a tubular neighborhood $\nu(S_i)$ has boundary $S^3$, and the projection to the core $\partial\nu(S_i)\to S_i$ is the Hopf fibration. Therefore, removing all $\nu(S_i)$ and tubing each $\partial\nu(S_i)$ to $\{1\}\times(S^1\times S^2)\subset\overline{Z^t}$ give a cobordism $\Sigma_0^t$ from $L_1$ to $L_0\sqcup(\sqcup_iT(n_i,n_i)_{p_i,q_i})$ in $I\times\ell(S^1\times S^2)$, where $n_i=p_i+q_i$. Topologically, $\Sigma_0^t$ is obtained by deleting $\sum_i(p_i+q_i)$ disks in the interior of $\Sigma^t\cong\Sigma$. Now by \cite[Theorem~1.5,Proposition~3.7]{manolescu2023generalization} we have $$\left(s(L_0)+\sum_{i=1}^ks(T(n_i,n_i)_{p_i,q_i})-k\right)-s(L_1)\ge\chi(\Sigma_0^t)=\chi(\Sigma)-\sum_{i=1}^k(p_i+q_i).$$ By Theorem~\ref{thm:s} with $m=n=n_i$, this simplifies to 
\begin{align*}
&s(L_0)-s(L_1)\ge\chi(\Sigma)-\sum_{i=1}^k((|p_i-q_i|-1)^2-2\min(p_i,q_i)-1+p_i+q_i)\\
=&\,\chi(\Sigma)-\sum_{i=1}^k|p_i-q_i|^2+\sum_{i=1}^k|p_i-q_i|=\chi(\Sigma)+[\Sigma]^2+|[\Sigma]|'.\qedhere
\end{align*}
\end{proof}

\begin{proof}[Proof of Theorem~\ref{thm:s}]
We induct on $m+n$. The base cases are $m=n$, which we have already addressed. To perform the induction step, by symmetry we may assume $n<m$. If $pq=0$ we conclude as before. Assume $p,q>0$, thus $d\ge2$. There is a cobordism $\Sigma$ from $T(n,m-n)_{p,q}$ to $T(n,m)_{p,q}$ in $\overline{\CP^2}$ obtained by adding a positive full twist. The surface $\Sigma$ is a disjoint union of $d$ annuli, which intersects a copy of $\overline{\CP^1}\subset\overline{\CP^2}$ transversely at $n_1p$ points positively and $n_1q$ points negatively. Applying Corollary~\ref{cor:adjunction1} to $\Sigma$, we obtain $$s(T(n,m)_{p,q})\le s(T(n,m-n)_{p,q})+n_1^2|p-q|^2-n_1|p-q|=(n_1|p-q|-1)(m_1|p-q|-1)-2\min(p,q).$$ On the other hand, since $T(n,m)$ is a $d$-cable on $T(n_1,m_1)$, there is a saddle cobordism $T(n-2n_1,m-2m_1)_{p-1,q-1}\sqcup U\to T(n,m)_{p,q}$. Thus $$s(T(n,m)_{p,q})\ge s(T(n-2n_1,m-2m_1)_{p-1,q-1}\sqcup U)-1=(n_1|p-q|-1)(m_1|p-q|-1)-2\min(p,q),$$ which completes the proof.
\end{proof}

\section{Applications of the adjunction inequality}\label{sec:adjunction}
\subsection{Eventual linearity of the \texorpdfstring{$s$}{s}-invariant under full twists}\label{sec:s_sb}
Let $L$ be an oriented link in $S^3$. A full twist can be performed to $L$ along any (unoriented) $2$-disk in $S^3$ that intersects $L$ transversely in the interior. More generally, if $D_1,\cdots,D_\ell\subset S^3$ are $\ell$ disjoint such $2$-disks, we can independently perform any number of full twists along these disks. Given $\vec n=(n_1,\cdots,n_\ell)$, let $L(D_1,D_2,\cdots,D_\ell;\vec n)\subset S^3$ denote the oriented link obtained by performing $n_i$ full twists ($n_i$ positive ones if $n_i\ge0$; $-n_i$ negative ones if $n_i<0$) to $L$ along $D_i$. Let $d_i$ denote the algebraic intersection number of $D_i$ and $L$ (which is well-defined up to sign).

\begin{Prop}\label{prop:linearity}
For $n_1,\cdots,n_\ell$ large, the number $$s(L;D_1,\cdots,D_\ell;\vec n):=s(L(D_1,\cdots,D_\ell;\vec n))-\sum_{i=1}^\ell n_i|d_i|(|d_i|-1)$$ is independent of $n_1,\cdots,n_\ell$.
\end{Prop}

Therefore, the stable number, denoted $s(L;D_1,\cdots,D_\ell)$, is an isotopy invariant of the oriented link $L$ together with $2$-disks $D_1,\cdots,D_\ell$ in $S^3$. In the special case $n_1=n_2=\cdots=n_\ell$, this answers positively a question of Manolescu-Marengon-Sarkar-Willis \cite[Question~9.1]{manolescu2023generalization}.

\begin{Rmk}
\begin{enumerate}[(1)]
\item By performing $0$-surgeries on each $\partial D_i\subset S^3$, $L$ can be regarded as an oriented link in $\ell(S^1\times S^2)$. Manolescu-Marengon-Sarkar-Willis \cite[Theorem~1.4]{manolescu2023generalization} proved Proposition~\ref{prop:linearity} in the special case $d_1=\cdots=d_\ell=0$, $n_1=\cdots=n_\ell$, and further showed that in this case the stable number $s(L;D_1,\cdots,D_\ell)$ is an invariant of the null-homologous oriented link $L$ in $\ell(S^1\times S^2)$, which can then be defined as the $s$-invariant of $L\subset\ell(S^1\times S^2)$. However, as followed from their Remark~9.6, the stable number in the general case does not define an invariant of $L\subset\ell(S^1\times S^2)$. In fact, sliding the strands intersecting $D_i$ over $\partial D_i$ changes the stable number by $\pm2|d_i|(|d_i|-1)$. It would be interesting if one can use Proposition~\ref{prop:linearity} to define $s$-invariants of links in some other $3$-manifolds, or to define $s$-invariants for links in $\ell(S^1\times S^2)$ valued in $\Z/gcd(d_1,\cdots,d_\ell)$.
\item Manolescu-Marengon-Sarkar-Willis \cite[Conjecture~8.31]{manolescu2023generalization} conjectured that when $d_1=\cdots=d_\ell=0$, if there is a generic projection of $L,D_1,\cdots,D_\ell$ to $\R^2$ which consists of $k$-disjoint $2$-disks corresponding to $D_i$'s and a positive link diagram corresponding to $L$, then the number $s(L;D_1,\cdots,D_\ell;\vec n)$ already stabilizes when $\vec n=\vec0$, i.e. it is independent of $\vec n$ for $n_1,\cdots,n_\ell\ge0$. This is not true, because adding a positive full twist appropriately to two oppositely oriented strands in the right-handed trefoil unknots it, but the right-handed trefoil and the unknot have $s$-invariants $2$ and $0$, respectively.
% \item The only properties we will use about the $s$-invariant in the proof below are the adjunction inequality (Corolloary~\ref{cor:adjunction1}) and the Kawamura-Lobb inequality \cite[Theorem~1.5]{kawamura2015estimate}. Since these formal properties also hold for (a renormalization of) the $\tau$-invariant for links \cite{hedden2023knot}\cite[Theorem~1.4]{cavallo2020slice}, Proposition~\ref{prop:linearity} holds with $s$ replaced by $2\tau$.
\end{enumerate}
\end{Rmk}\smallskip

The proof of Proposition~\ref{prop:linearity} is an easy consequence of the following bound on the behavior of $s$-invariant under twists.
\begin{Prop}\label{prop:s_twists}
Let $L\subset S^3$ be an oriented link, and $D\subset S^3$ a $2$-disk intersecting it transversely in the interior, in $p$ points positively and $q$ points negatively, where $p\ge q$. Then for $m'>m$, $$s(L;D;m')-s(L;D;m)\in\begin{cases}[-2p+2,0],&p>q\\
[-2p,0],&p=q.\end{cases}$$
\end{Prop}
\begin{Rmk}
By considering $T(n,-2n)_{p,q}$, $T(n,-n)_{p,q}$, and a disjoint union of $n$ unknots, we see the bounds in Proposition~\ref{prop:s_twists} are sharp (cf. Corollary~\ref{cor:bar_s}).
\end{Rmk}
\begin{proof}
The lower bounds are exactly those in \cite[Theorem~1.2]{roberts2011extending}. We prove the upper bounds. Adding $m'-m$ full twists along $D$ gives a cobordism $L(D;m)\to L(D;m')$ in $(m'-m)\overline{\mathbb{CP}^2}$ with Euler characteristic $0$ and homology class $(p-q,\cdots,p-q)$. Thus Corollary~\ref{cor:adjunction1} gives $$s(L(D;m'))\le s(L(D;m))+(m-m')((p-q)^2-(p-q)),$$ or equivalently $s(L;D;m')\le s(L;D;m)$.
\end{proof}

\begin{proof}[Proof of Proposition~\ref{prop:linearity}]
Proposition~\ref{prop:s_twists} implies $s(L;D_1,\cdots,D_\ell;\vec n)$ is nonincreasing in the coordinates of $\vec n$, and has a lower bound independent of $\vec n$; consequently, it is constant for large $\vec n$.
\end{proof}

\subsection{Knots with large \texorpdfstring{$\CP^2$}{CP2}- and \texorpdfstring{$\overline{\CP^2}$}{(-CP2)}-genus}\label{sec:s_genus}
For a closed oriented smooth $4$-manifold $M$ and a knot $K$ in $S^3$, the $M$-genus of $K$ is the minimal genus of a smooth orientable surface in $M\backslash B^4$ that bounds $K$. Marco Marengon informed us the following consequence of Corollary~\ref{cor:adjunction} (see also \cite[Proposition~2.1]{marengon2022note}).
\begin{Prop}\label{prop:CP2_genus}
There exist knots with simultaneously arbitrarily large $\CP^2$-genus and $\overline{\CP^2}$-genus.
\end{Prop}
\begin{pf}
Following the argument in \cite[Proposition~2.1]{marengon2022note}, after Corollary~\ref{cor:adjunction}, it suffices to construct a knot with large $s$-invariant, small $\tau$-invariant, and vanishing Levine-Tristram signature function. For example, such a knot can be taken to be a connected sum of some copies of the untwisted negative Whitehead double of $T(2,-3)$ (which has vanishing Levine-Tristram signature \cite[Theorem~2]{litherland2006signatures} and $\tau=-1$ \cite[Theorem~1.2]{hedden2008ozsvath}) with some copies of Piccirillo's companion to the Conway knot, denoted as $K'$ in \cite{piccirillo2020conway} (which has vanishing Levine-Tristram signature and $\tau=0$ as it is topologically slice, and $s=2$).
\end{pf}
This is in contrast to the fact that the $(\CP^2\#\overline{\CP^2})$-genus (and the $(S^2\times S^2)$-genus) of any knot is $0$ \cite[Corollary~3 and Remark]{norman1969dehn}, and the fact that every knot has topological $\CP^2$-genus and $\overline{\CP^2}$-genus at most $1$ \cite[Corollary~1.15]{kasprowski2022embedding}. We remark that the proof does not carry to $k\CP^2$ for $k>1$, and to our knowledge there is currently no knot known to be non-slice in $2\CP^2$. In fact all knots are topologically slice in $2\CP^2$ \cite[Corollary~1.15]{kasprowski2022embedding}.

\section{The Lee filtration structure of \texorpdfstring{$T(n,m)$}{T(n,m)}}\label{sec:filtration}
In this section we prove Theorem~\ref{thm:filtration}.

By Lee \cite{lee2005endomorphism} and Rasmussen \cite{rasmussen2010khovanov}, the Lee homology of $T(n,m)$ as a vector space is spanned by the canonical generators $[s_\mathfrak{o}]$, one for each orientation $\mathfrak{o}$ of $T(n,m)$. For our purpose in the next subsection, it will be convenient to use the rescaled canonical generators $[\widetilde{s_\mathfrak{o}}]$ as defined by Rasmussen \cite{rasmussen2005khovanov}.

We identify a rescaled canonical generator $[s_\mathfrak{o}]$ with the orientation $\mathfrak{o}$. Upon labeling the $d$ components of $T(n,m)$ by $1,2,\cdots,d$ and choosing a preferred direction, we further identify $\mathfrak{o}$ with a subset of $[d]:=\{1,2,\cdots,d\}$. Thus, $Kh_{Lee}(T(n,m))$ is identified with $\Q\{2^{[d]}\}$, where $2^{[d]}$ denotes the power set of $[d]$. Since the linking number between any two different components of $T(n,m)$ is $n_1m_1$, under this identification, the span of subsets of $[d]$ with cardinality $k$ or $d-k$ is identified with the homological degree $2n_1m_1k(d-k)$ part of the Lee homology \cite[Proposition~4.3]{lee2005endomorphism}\cite{rasmussen2010khovanov}.

Hence, the Lee homology of $T(n,m)$ is zero for homological degrees not equal to $2n_1m_1pq$ for any $p+q=d$. For each $p+q=d$ we have
\begin{equation}\label{eq:Lee=power_set}
Kh_{Lee}^{2n_1m_1pq}(T(n,m))\cong\Q\{X\subset[d]\colon\#X=p\text{ or }q\}.
\end{equation}
It remains to prove for each $p+q=d$, $gr(Kh_{Lee}^{2n_1m_1pq}(T(n,m)))$ has the desired graded dimension. By symmetry we assume throughout that $p\ge q$.

\subsection{Halve the dimensions}\label{sec:filtration_half}
Let $Kh_{Lee,0}^{2n_1m_1pq}(T(n,m))$ (resp. $Kh_{Lee,1}^{2n_1m_1pq}(T(n,m))$) denote the subspace of $Kh_{Lee}^{2n_1m_1pq}(T(n,m))$ spanned by subsets of $[d]$ with cardinality $q$ (resp. $p$), namely,
\begin{equation}\label{eq:Lee_0=power_set}
Kh_{Lee,0}^{2n_1m_1pq}(T(n,m))\cong\Q\{X\subset[d]\colon\#X=q\}.
\end{equation}
Thus when $p=q$, we simply have $Kh_{Lee,0}^{2n_1m_1pq}(T(n,m))=Kh_{Lee,1}^{2n_1m_1pq}(T(n,m))=Kh_{Lee}^{2n_1m_1pq}(T(n,m))$.

\begin{Lem}\label{lem:half}
For $p>q$, as graded vector spaces we have
$$gr(Kh_{Lee}^{2n_1m_1pq}(T(n,m)))\cong gr(Kh_{Lee,0}^{2n_1m_1pq}(T(n,m)))\otimes(\Q\oplus\Q\{2\}).$$
\end{Lem}

\begin{pf}
The Lee homology $Kh_{Lee}^{2n_1m_1pq}(T(n,m))$ has a quantum $\Z/4$-grading, and the elements are supported in odd gradings if $d$ is even, and even gradings if $d$ is odd. In either case, we can write $Kh_{Lee}^{2n_1m_1pq}(T(n,m))=V_1\oplus V_2$ as a direct sum of two $\Z/4$-homogeneous components. Let $\iota$ be an involution on $Kh_{Lee}(T(n,m))$ which is $1$ on $V_1$ and $-1$ on $V_2$. Then $\iota$ maps every $[\widetilde{s_{\mathfrak{o}}}]$ to $\pm[\widetilde{s_{\bar{\mathfrak{o}}}}]$ \cite[Lemma~3.5]{rasmussen2010khovanov}, where $\bar{\mathfrak{o}}$ denotes the reverse orientation of $\mathfrak{o}$. Thus, it interchanges the two subspaces $Kh_{Lee,0/1}^{2n_1m_1pq}(T(n,m))$.

Let $q\colon Kh_{Lee}^{2n_1m_1pq}(T(n,m))\to\Z\sqcup\{+\infty\}$ denote the quantum filtration degree function. Then $q(x)=q(\iota x)=\min(q(x+\iota x),q(x-\iota x))$ for nonzero $x$. We claim that $q(x)=\max(q(x+\iota x),q(x-\iota x))-2$ if $x\in Kh_{Lee,0}^{2n_1m_1pq}(T(n,m))$. This would imply the desired statement, for example by performing an induction from the top filtration degree.

Let $X_i\colon Kh_{Lee}(T(n,m))\to Kh_{Lee}(T(n,m))$ be the map induced by putting a dot on the $i^{\text{th}}$ component of $T(n,m)$ (see \cite[Section~11.2]{bar2005khovanov}), $1\le i\le d$. Then $X_i$ has quantum filtration degree $-2$. For a suitable choice of sign for $X_i$, $X_i[\widetilde{s_{\mathfrak{o}}}]=\epsilon[\widetilde{s_{\mathfrak{o}}}]$ where $\epsilon=1$ if $i\in\mathfrak{o}$ and $-1$ otherwise. It follows that for $x\in Kh_{Lee,0}^{2n_1m_1pq}(T(n,m))$ we have $(\sum_{i=1}^dX_i)x=(q-p)x$ and $(\sum_{i=1}^dX_i)\iota x=(p-q)\iota x$. Hence, $(\sum_{i=1}^dX_i)/(q-p)$ is a map of quantum filtration degree $-2$ that interchanges $x\pm\iota x$. The claim follows.
\end{pf}

\subsection{An \texorpdfstring{$S_d$}{Sd}-symmetry}
The torus link $T(n,m)$ can be seen as a $d$-cable of the torus knot $T(n_1,m_1)$. Thus, every element in the braid group $B_d$ lifts to an isotopy from $T(n,m)$ to itself. This induces a $B_d$-action on $Kh_{Lee}(T(n,m))$ up to sign, which respects the homological grading, the quantum $\Z/4$-grading, and the quantum filtration. Our choice of using the rescaled canonical generators $[\widetilde{s_\mathfrak{o}}]$ has the advantage that an element $\alpha\in B_d$ acts by $\alpha\cdot[\widetilde{s_\mathfrak{o}}]=\pm[\widetilde{s_{\bar\alpha\cdot\mathfrak{o}}}]$ \cite[Proposition~3.2]{rasmussen2005khovanov}, where $\bar\alpha$ denote the image of $\alpha$ in $S_d$, which acts on the power set $2^{[d]}$ in the natural way. In particular, we see the $B_d$-action descends to an $S_d$-action, up to sign. In fact, following Grigsby-Licata-Wehrli \cite[Theorem~2]{grigsby2018annular}, we can fix a sign convention (i.e. a choice of signs for the action of each $\alpha\in B_d$ on $Kh_{Lee}(T(n,m))$) to remove the sign ambiguity (cf. proof of Proposition~\ref{prop:equivariant}).

\begin{Prop}\label{prop:Lee_irr_rep}
As $S_d$-representations over $\Q$, 
\begin{equation}\label{eq:Lee_irr_rep}
Kh_{Lee,0}^{2n_1m_1pq}(T(n,m))\cong\bigoplus_{r=0}^q(d-r,r).
\end{equation}
\end{Prop}
\begin{pf}
By the proceeding paragraph, upon changing the signs of some $[\widetilde{s_\mathfrak{o}}]$'s, the identification $Kh_{Lee}(T(n,m))\cong\Q\{2^{[d]}\}$ is $S_d$-equivariant. Now the statement follows by restricting to \eqref{eq:Lee_0=power_set} and using the standard fact that $\Q\{X\subset[d]\colon\#X=q\}\cong\bigoplus_{r=0}^q(d-r,r)$ as $S_d$-representations, see e.g. \cite{123754}.
\end{pf}

\begin{Rmk}
It is worth mentioning that Grigsby-Licata-Wehrli \cite[Theorem~2]{grigsby2018annular} actually showed that the $S_d$-action on $Kh_{Lee}(T(n,m))$ descends to an action of the Temperley-Lieb algebra $TL_d(1)$. In fact, the irreducible $S_d$-representations $(d-r,r)$ are exactly the ones pullbacked from irreducible $TL_d(1)$-representations.
\end{Rmk}

The Lee homology of the unknot $U$ is generated by the two rescaled canonical generators, denoted $A,B$. In standard notation of the Lee homology defined via the Frobenius algebra $\Q[X]/(X^2-1)$ we can take $A=(X+1)/2$ and $B=(X-1)/2$. Define an involution $\iota$ on $Kh_{Lee}(U)$ by $A\mapsto B$, which equips $Kh_{Lee}(U)$ with a $\Z/2$-action. Then $Kh_{Lee}(U)\cong1\oplus\epsilon$ as $\Z/2$-representations. Abuse the notation and use $1,\epsilon$ to also denote the corresponding subrepresentations of $Kh_{Lee}(U)$. The quantum filtration structure of $Kh_{Lee}(U)$ is determined by
\begin{equation}\label{eq:U_filt}
q(\epsilon)=1,\ q(1)=-1.
\end{equation}

From the description of $T(n,m)$ as a $d$-cable on $T(n_1,m_1)$, we see when $d\ge2$ there is a saddle cobordism $$T(n-2n_1,m-2m_1)\sqcup U\to T(n,m).$$

\begin{Prop}\label{prop:equivariant}
The induced map 
\begin{equation}\label{eq:saddle}
Kh_{Lee}(T(n-2n_1,m-2m_1)\sqcup U)\to Kh_{Lee}(T(n,m))
\end{equation}
by the saddle cobordism is $(S_{d-2}\times\Z/2)$-equivariant. Here $S_{d-2}\times\Z/2=S_{d-2}\times S_2\subset S_d$ via the natural inclusion. Moreover, upon negating the involution on $Kh_{Lee}(U)$, the induced map
\begin{equation}\label{eq:back_saddle}
Kh_{Lee}(T(n,m))\to Kh_{Lee}(T(n-2n_1,m-2m_1)\sqcup U)
\end{equation}
by the (backward) saddle cobordism is $(S_{d-2}\times\Z/2)$-equivariant.
\end{Prop}
\begin{pf}
We only prove the equivariance of \eqref{eq:saddle}. The equivariance of \eqref{eq:back_saddle} is proved similarly.

Let $\Sigma$ denote the saddle cobordism $T(n-2n_1,m-2m_1)\sqcup U\to T(n,m)$. For a braid $\alpha\in B_d$, let $\Sigma_\alpha$ denote the corresponding self-isotopy of $T(n,m)$. By an explicit calculation, one can show the involution $\iota$ on $Kh_{Lee}(U)$ is the map induced by the self-isotopy $\Sigma_\iota$ of $U$ flipping itself around. Now the following pairs of cobordisms that are isotopic rel boundary
$$\Sigma\circ\Sigma_{\sigma_i}\sim\Sigma_{\sigma_i}\circ\Sigma,\ i=1,2,\cdots,d-3;\ \Sigma\circ\Sigma_\iota\sim\Sigma_{\sigma_{d-1}}\circ\Sigma$$
show the equivariance of \eqref{eq:saddle} up to sign. Here $\sigma_1,\cdots,\sigma_{d-1}$ are the usual braid group generators.

To remove the sign ambiguity, we have to look into the sign convention in \cite[Section~7.2]{grigsby2018annular} which we adopted. Let $\mathfrak{o}_p$ denote the parallel orientation $\emptyset\subset[d]$. The sign of $\varphi_i:=Kh_{Lee}(\Sigma_{\sigma_i})$ is chosen so that $\varphi_i([\widetilde{s_{\mathfrak{o}_p}}])=[\widetilde{s_{\mathfrak{o}_p}}]$. This is the sign convention we adopted to define $Kh_{Lee}(T(n,m))$ as an $S_d$-representation.

Grigsby-Licata-Wehrli \cite{grigsby2018annular} also considered maps $\psi_i$ on $Kh_{Lee}(T(n,m))$, each induced by the annular cobordism $T(n,m)\to T(n-2n_1,m-2m_1)$ that annihilates the components labeled $i,i+1$, followed by the annular cobordism $T(n-2n_1,m-2m_1)\to T(n,m)$ that recreates two components with labels $i,i+1$. This defines $\psi_i$ up to sign. Let $\mathfrak{o}_a$ denote the alternating orientation $\{2,4,\cdots,2\lfloor d/2\rfloor\}\subset[d]$, and $\mathfrak{o}_{a,i}$ denote the symmetric difference between $\mathfrak{o}_a$ and $\{i,i+1\}$. Then $\psi_i([\widetilde{s_{\mathfrak{o}_a}}])=\epsilon_1[\widetilde{s_{\mathfrak{o}_a}}]+\epsilon_2[\widetilde{s_{\mathfrak{o}_{a,i}}}]$ for some signs $\epsilon_1,\epsilon_2\in\{\pm1\}$ (\cite[Proposition~3.2]{rasmussen2005khovanov}). The sign of $\psi_i$ is fixed by demanding
\begin{equation}\label{eq:sign}
\psi_i([\widetilde{s_{\mathfrak{o}_a}}])=-[\widetilde{s_{\mathfrak{o}_a}}]\pm[\widetilde{s_{\mathfrak{o}_{a,i}}}].
\end{equation}
Under these two sign conventions, they showed that $\varphi_i=\mathrm{id}+\psi_i$ \cite[Proposition~9]{grigsby2018annular}.

The above sign fixes do not depend on the signs of the rescaled canonical generators $[\widetilde{s_{\mathfrak{o}}}]$, but our assumption (in the proof of Proposition~\ref{prop:Lee_irr_rep}) that $Kh_{Lee}(T(n,m))\cong\Q\{2^{[d]}\}$ is $S_d$-equivariant does. Since $\varphi_i([\widetilde{s_{\mathfrak{o}_a}}])=[\widetilde{s_{\mathfrak{o}_a}}]-[\widetilde{s_{\mathfrak{o}_a}}]\pm[\widetilde{s_{\mathfrak{o}_{a,i}}}]=\pm[\widetilde{s_{\mathfrak{o}_{a,i}}}]$, the sign in \eqref{eq:sign} is $+$ by our convention.

Since we have shown that \eqref{eq:saddle} is equivariant up to sign, to prove the full equivariance it now suffices to check on particular Lee generators.

First we check $Kh_{Lee}(\Sigma)\varphi_i=\varphi_iKh_{Lee}(\Sigma)$, $1\le i\le d-3$. Since $\psi_i$ vanishes on any $[\widetilde{s_{\mathfrak{o}}}]$ where the $i,i+1$ components are parallel in $\mathfrak{o}$, we see $\varphi_i=\mathrm{id}$ on such generators. Therefore $Kh_{Lee}(\Sigma)\varphi_i([\widetilde{s_{\mathfrak{o}}}]\otimes A)=Kh_{Lee}(\Sigma)([\widetilde{s_{\mathfrak{o}}}]\otimes A)=\varphi_iKh_{Lee}(\Sigma)([\widetilde{s_{\mathfrak{o}}}]\otimes A)$ (which is nonzero) for any such $\mathfrak{o}$.

Next we check $Kh_{Lee}(\Sigma)\iota=\varphi_{d-1}Kh_{Lee}(\Sigma)$. By definition, $\psi_{d-1}$ factors through $Kh_{Lee}(\Sigma)$, and we see $\psi_{d-1}([\widetilde{s_{\mathfrak{o}_a}}])=Kh_{Lee}(\Sigma)([\widetilde{s_{\mathfrak{o}_a}}]\otimes1)$ up to sign. In view of \eqref{eq:sign} and noting $1=A-B$ in $Kh_{Lee}(U)$, upon switching $A,B$ we may assume $Kh_{Lee}(\Sigma)([\widetilde{s_{\mathfrak{o}_a}}]\otimes A)=\pm[\widetilde{s_{\mathfrak{o}_a}}]$ and $Kh_{Lee}(\Sigma)([\widetilde{s_{\mathfrak{o}_a}}]\otimes B)=\pm[\widetilde{s_{\mathfrak{o}_{a,d-1}}}]$, where the two signs $\pm$ are equal. It follows that $$Kh_{Lee}(\Sigma)\iota([\widetilde{s_{\mathfrak{o}_a}}]\otimes A)=Kh_{Lee}(\Sigma)([\widetilde{s_{\mathfrak{o}_a}}]\otimes B)=\pm[\widetilde{s_{\mathfrak{o}_{a,d-1}}}]=\pm\varphi_{d-1}([\widetilde{s_{\mathfrak{o}_a}}])=\varphi_{d-1}Kh_{Lee}(\Sigma)([\widetilde{s_{\mathfrak{o}_a}}]\otimes A).$$ We are done.
\end{pf}

\subsection{Proof of Theorem~\ref{thm:filtration}}
By Lemma~\ref{lem:half}, Theorem~\ref{thm:filtration} reduces to showing that nonzero components of $gr(Kh_{Lee,0}^{2n_1m_1pq}(T(n,m)))$ are determined by
\begin{equation}\label{eq:half_filtration}
\dim gr(Kh_{Lee,0}^{2n_1m_1pq}(T(n,m)))^{6n_1m_1pq+s(T(n,m)_{p,q})+2r-1}=\dim(d-r,r),\ r=0,1,\cdots,q.
\end{equation}

Let $V_{d-r,r}^{2n_1m_1pq}\subset Kh_{Lee,0}^{2n_1m_1pq}(T(n,m))$ denote the irreducible $S_d$-subrepresentation that corresponds to $(d-r,r)$ via \eqref{eq:Lee_irr_rep}. Since the $S_d$-action respects the quantum filtration structure, all nonzero elements in $V_{d-r,r}^{2n_1m_1pq}$ have the same filtration degree, denoted $q(V_{d-r,r}^{2n_1m_1pq})$. Moreover, since every irreducible $S_d$-representation appears at most once in $Kh_{Lee,0}^{2n_1m_1pq}(T(n,m))$, we conclude that $gr(Kh_{Lee,0}^{2n_1m_1pq}(T(n,m)))\cong\bigoplus_{r=0}^qV_{d-r,r}^{2n_1m_1pq}$ as graded vector spaces. Now \eqref{eq:half_filtration} reduces to showing that
\begin{equation}\label{eq:q_V}
q(V_{d-r,r}^{2n_1m_1pq})=6n_1m_1pq+s(T(n,m)_{p,q})+2r-1,\ r=0,1,\cdots,q.
\end{equation}

We proceed by induction on $d$. The case $d=0$ is plain, and the case $d=1$ follows directly from Theorem~\ref{thm:s}. Below we assume $d\ge2$.

In the case $q>0$, the map \eqref{eq:saddle} restricts to \begin{equation}\label{eq:Phi}
\Phi\colon Kh_{Lee,0}^{2n_1m_1(p-1)(q-1)}(T(n-2n_1,m-2m_1))\otimes Kh_{Lee}(U)\to Kh_{Lee,0}^{2n_1m_1pq}(T(n,m)),
\end{equation}
where the left hand side is isomorphic to $\bigoplus_{r=0}^{q-1}((d-r-2,r)\otimes(1\oplus\epsilon))$ as $(S_{d-2}\times\Z/2)$-representations. Since $$\mathrm{Res}^{S_d}_{S_{d-2}\times\Z/2}(d-r,r)=\begin{cases}(d-r,r-2)\otimes1\oplus(d-r-1,r-1)\otimes(1\oplus\epsilon)\oplus(d-r-2,r)\otimes1,&r<\frac d2\\(d-r,r-2)\otimes1\oplus(d-r-1,r-1)\otimes\epsilon,&r=\frac d2,\end{cases}$$ (here a non-Young-diagram $(a,b)$ is considered to be zero) the right hand side of \eqref{eq:Phi} contains a unique copy of $(d-r-2,r)\otimes\epsilon$ for every $0\le r\le q-1$, denoted $W_{d-r-1,r+1}^{2n_1m_1pq}$, which is a subspace of $V_{d-r-1,r+1}^{2n_1m_1pq}$.

By an explicit description of cobordism maps on Lee homology in terms of the canonical generators (see \cite[Proposition~3.2]{rasmussen2005khovanov}), \eqref{eq:saddle} is injective. Consequently $\Phi$ is injective, thus by Proposition~\ref{prop:equivariant}, it maps $V_{d-r-2,r}^{2n_1m_1(p-1)(q-1)}\otimes\epsilon$ isomorphically onto $W_{d-r-1,r+1}^{2n_1m_1pq}$. Similarly, the backward map $$\Psi\colon Kh_{Lee,0}^{2n_1m_1pq}(T(n,m))\to Kh_{Lee,0}^{2n_1m_1(p-1)(q-1)}(T(n-2n_1,m-2m_1))\otimes Kh_{Lee}(U)$$ is surjective and maps $W_{d-r-1,r+1}^{2n_1m_1pq}$ isomorphically onto $V_{d-r-2,r}^{2n_1m_1(p-1)(q-1)}\otimes1$. Upon shifting $Kh_{Lee,0}(T(n-2n_1,m-2m_1))\otimes Kh_{Lee}(U)$ by $[2n_1m_1(d-1)]\{6n_1m_1(d-1)\}$ to account for the bidegree differences between $Kh(T(n,m))$, $Kh(T(n,m)_{p,q})$ and between $Kh(T(n-2n_1,m-2m_1))$, $Kh(T(n-2n_1,m-2m_1)_{p-1,q-1})$, the maps $\Phi,\Psi$ preserve the homological degree, and have quantum filtration degree $-1$. By \eqref{eq:U_filt} and induction hypothesis, we conclude that
\begin{align*}
q(V_{d-r-1,r+1}^{2n_1m_1pq})=&\,q(W_{d-r-1,r+1}^{2n_1m_1pq})=q(V_{d-r-2,r}^{2n_1m_1(p-1)(q-1)})+6n_1m_1(d-1)\\
=&\,6n_1m_1pq+s(T(n,m)_{p,q})+2r+1.
\end{align*}
This proves \eqref{eq:q_V} for $r\ne0$. Finally, Theorem~\ref{thm:s} implies $q(V_{d-r,r}^{2n_1m_1pq})=6n_1m_1pq+s(T(n,m)_{p,q})-1$ for some $r$, which is now necessarily $0$.\qed

\section{Proof of Theorem~\ref{thm:lower_bound}}\label{sec:proof}
We follow the induction scheme set up by Sto\v{s}i\'c \cite{stovsic2007homological,stovsic2009khovanov}. For ease of notation, in this section, we write $T_{n,m}$ for the torus link $T(n,m)$. Define an auxiliary family of links $D_{n,m}^i$, $m,n\ge0$, $0\le i\le n-1$, as the braid closure of the braid $(\sigma_1\cdots\sigma_{n-1})^m\sigma_1\cdots\sigma_i\in B_n$. Thus, $T_{n,m}=D_{n,m}^0=D_{n,m-1}^{n-1}$. For $i>0$, performing a $0$-resolution to the crossing of $D_{n,m}^i$ corresponding to the last letter $\sigma_i$ gives the link $D_{n,m}^{i-1}$, while performing a $1$-resolution gives another link which we denote by $E_{n,m}^{i-1}$. The reader is warned these notations do not agree with those in \cite{stovsic2007homological,stovsic2009khovanov}.

For our purpose, the cases $m=n,n-1$ will be useful. The following statement is easily checked. The first two items appeared in \cite[Proof of Theorem~1]{stovsic2009khovanov}. See Figure~\ref{fig:R_moves} for an illustration of the third item. 

\begin{Lem}\label{lem:E}\mbox{}\vspace{-5pt}
\begin{itemize}
\item $E_{n,n-1}^{n-2}\simeq D_{n-2,n-3}^{n-3}\sqcup U$;
\item $E_{n,n-1}^i\simeq D_{n-2,n-3}^i$, $i=0,1,\cdots,n-3$;
\item $E_{n,n}^i\simeq D_{n-2,n-2}^{i-1}$, $i=1,2,\cdots,n-2$;
\item $E_{n,n}^0\simeq D_{n-2,n-2}^0\sqcup U$.\qed
\end{itemize}
\end{Lem}

\begin{figure}
\centering
\scalebox{.8}{%% Creator: Inkscape 1.2.2 (732a01da63, 2022-12-09), www.inkscape.org
%% Accompanies image file 'E_5_5_3_to_D_3_3_1.pdf' (pdf, eps, ps)
\begingroup%
  \makeatletter%
  \providecommand\color[2][]{%
    \errmessage{(Inkscape) Color is used for the text in Inkscape, but the package 'color.sty' is not loaded}%
    \renewcommand\color[2][]{}%
  }%
  \providecommand\transparent[1]{%
    \errmessage{(Inkscape) Transparency is used (non-zero) for the text in Inkscape, but the package 'transparent.sty' is not loaded}%
    \renewcommand\transparent[1]{}%
  }%
  \providecommand\rotatebox[2]{#2}%
  \newcommand*\fsize{\dimexpr\f@size pt\relax}%
  \newcommand*\lineheight[1]{\fontsize{\fsize}{#1\fsize}\selectfont}%
  \ifx\svgwidth\undefined%
    \setlength{\unitlength}{498.46570663bp}%
    \ifx\svgscale\undefined%
      \relax%
    \else%
      \setlength{\unitlength}{\unitlength * \real{\svgscale}}%
    \fi%
  \else%
    \setlength{\unitlength}{\svgwidth}%
  \fi%
  \global\let\svgwidth\undefined%
  \global\let\svgscale\undefined%
  \makeatother%
  \begin{picture}(1,0.2192839)%
    \lineheight{1}%
    \setlength\tabcolsep{0pt}%
    \put(0,0){\includegraphics[width=\unitlength,page=1]{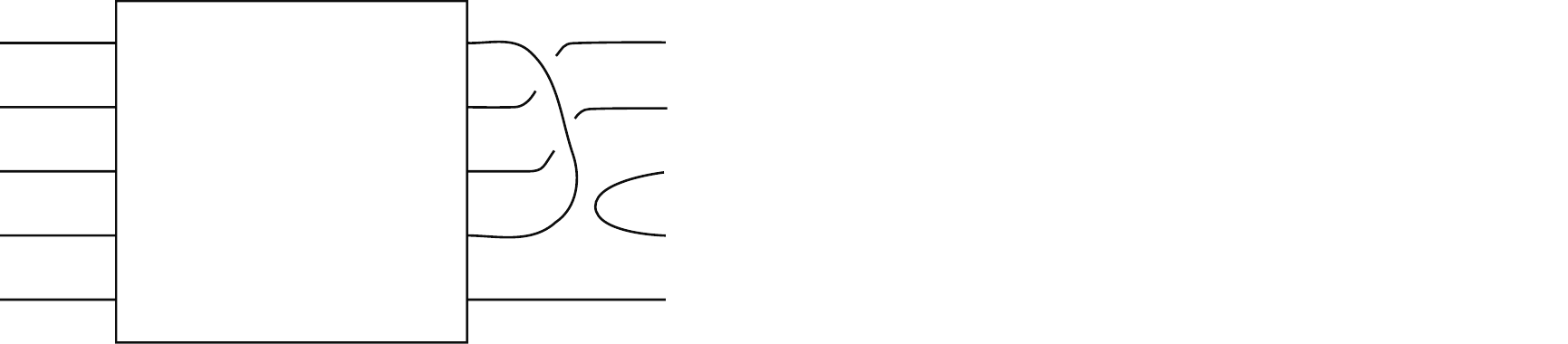}}%
    \put(0.1679926,0.10266415){\color[rgb]{0,0,0}\transparent{0.94560701}\makebox(0,0)[lt]{\lineheight{1.25}\smash{\begin{tabular}[t]{l}+1\end{tabular}}}}%
    \put(0,0){\includegraphics[width=\unitlength,page=2]{E_5_5_3_to_D_3_3_1.pdf}}%
    \put(0.7423535,0.10266419){\color[rgb]{0,0,0}\transparent{0.94560701}\makebox(0,0)[lt]{\lineheight{1.25}\smash{\begin{tabular}[t]{l}+1\end{tabular}}}}%
    \put(0,0){\includegraphics[width=\unitlength,page=3]{E_5_5_3_to_D_3_3_1.pdf}}%
  \end{picture}%
\endgroup%
}
\caption{The oriented link $E_{5,5}^3\simeq D_{3,3}^1$}
\label{fig:R_moves}
\end{figure}

Equip $D_{n,m}^i$ with the orientation where all components are oriented in the same direction. Equip $E_{n,m}^i$ with the orientation coming from the right hand sides of Lemma~\ref{lem:E} for $m=n,n-1$. Then all crossings in $D_{n,m}^i$ are positive, while $E_{n,n-1}^i$ has $2n-3$ negative crossings and $E_{n,n}^i$ has $2n-2$ negative crossings. Keeping track of the degree shifts (see e.g. \cite[Section~3.1 Case~II]{turner2017five}), the skein long exact sequences on Khovanov homology corresponding to the resolution at the last letter $\sigma_i$ in $D_{n,m}^i$ read
\begin{equation}\label{eq:LES_n}
\cdots\to Kh^{h-2n+2,q-6n+7}(E_{n,n-1}^{i-1})\to Kh^{h,q}(D_{n,n-1}^i)\to Kh^{h,q-1}(D_{n,n-1}^{i-1})\to\cdots,
\end{equation}
\begin{equation}\label{eq:LES_n-1}
\cdots\to Kh^{h-2n+1,q-6n+4}(E_{n,n}^{i-1})\to Kh^{h,q}(D_{n,n}^i)\to Kh^{h,q-1}(D_{n,n}^{i-1})\to\cdots.
\end{equation}
Moreover, the maps in the above long exact sequences are the maps induced by the obvious saddle cobordisms between the relevant links. Since the author is not aware of this last claim on cobordisms in the existing literature, we take a short detour and write this into the following lemma. 

\newsavebox{\firstmor}
\newsavebox{\secondmor}
\newsavebox{\thirdmor}
\savebox{\firstmor}{\scalebox{.15}{%% Creator: Inkscape 1.2.2 (732a01da63, 2022-12-09), www.inkscape.org
%% Accompanies image file 'crossing_to.pdf' (pdf, eps, ps)
\begingroup%
  \makeatletter%
  \providecommand\color[2][]{%
    \errmessage{(Inkscape) Color is used for the text in Inkscape, but the package 'color.sty' is not loaded}%
    \renewcommand\color[2][]{}%
  }%
  \providecommand\transparent[1]{%
    \errmessage{(Inkscape) Transparency is used (non-zero) for the text in Inkscape, but the package 'transparent.sty' is not loaded}%
    \renewcommand\transparent[1]{}%
  }%
  \providecommand\rotatebox[2]{#2}%
  \newcommand*\fsize{\dimexpr\f@size pt\relax}%
  \newcommand*\lineheight[1]{\fontsize{\fsize}{#1\fsize}\selectfont}%
  \ifx\svgwidth\undefined%
    \setlength{\unitlength}{127.60654918bp}%
    \ifx\svgscale\undefined%
      \relax%
    \else%
      \setlength{\unitlength}{\unitlength * \real{\svgscale}}%
    \fi%
  \else%
    \setlength{\unitlength}{\svgwidth}%
  \fi%
  \global\let\svgwidth\undefined%
  \global\let\svgscale\undefined%
  \makeatother%
  \begin{picture}(1,0.95478598)%
    \lineheight{1}%
    \setlength\tabcolsep{0pt}%
    \put(0,0){\includegraphics[width=\unitlength,page=1]{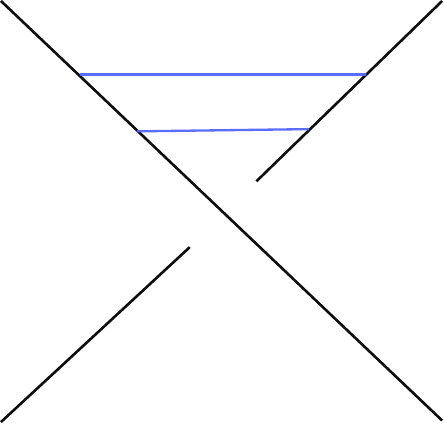}}%
  \end{picture}%
\endgroup%
}}
\savebox{\secondmor}{\scalebox{.15}{%% Creator: Inkscape 1.2.2 (732a01da63, 2022-12-09), www.inkscape.org
%% Accompanies image file '0-resolution_to.pdf' (pdf, eps, ps)
\begingroup%
  \makeatletter%
  \providecommand\color[2][]{%
    \errmessage{(Inkscape) Color is used for the text in Inkscape, but the package 'color.sty' is not loaded}%
    \renewcommand\color[2][]{}%
  }%
  \providecommand\transparent[1]{%
    \errmessage{(Inkscape) Transparency is used (non-zero) for the text in Inkscape, but the package 'transparent.sty' is not loaded}%
    \renewcommand\transparent[1]{}%
  }%
  \providecommand\rotatebox[2]{#2}%
  \newcommand*\fsize{\dimexpr\f@size pt\relax}%
  \newcommand*\lineheight[1]{\fontsize{\fsize}{#1\fsize}\selectfont}%
  \ifx\svgwidth\undefined%
    \setlength{\unitlength}{127.65746639bp}%
    \ifx\svgscale\undefined%
      \relax%
    \else%
      \setlength{\unitlength}{\unitlength * \real{\svgscale}}%
    \fi%
  \else%
    \setlength{\unitlength}{\svgwidth}%
  \fi%
  \global\let\svgwidth\undefined%
  \global\let\svgscale\undefined%
  \makeatother%
  \begin{picture}(1,0.96128906)%
    \lineheight{1}%
    \setlength\tabcolsep{0pt}%
    \put(0,0){\includegraphics[width=\unitlength,page=1]{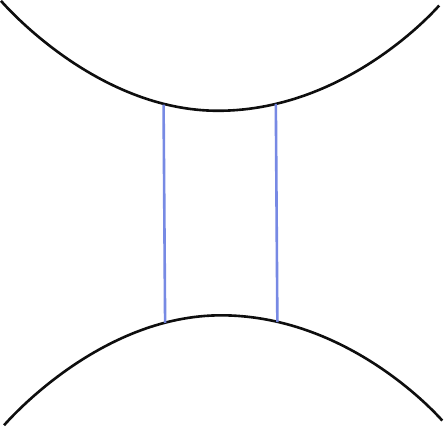}}%
  \end{picture}%
\endgroup%
}}
\savebox{\thirdmor}{\scalebox{.15}{%% Creator: Inkscape 1.2.2 (732a01da63, 2022-12-09), www.inkscape.org
%% Accompanies image file '1-resolution_to.pdf' (pdf, eps, ps)
\begingroup%
  \makeatletter%
  \providecommand\color[2][]{%
    \errmessage{(Inkscape) Color is used for the text in Inkscape, but the package 'color.sty' is not loaded}%
    \renewcommand\color[2][]{}%
  }%
  \providecommand\transparent[1]{%
    \errmessage{(Inkscape) Transparency is used (non-zero) for the text in Inkscape, but the package 'transparent.sty' is not loaded}%
    \renewcommand\transparent[1]{}%
  }%
  \providecommand\rotatebox[2]{#2}%
  \newcommand*\fsize{\dimexpr\f@size pt\relax}%
  \newcommand*\lineheight[1]{\fontsize{\fsize}{#1\fsize}\selectfont}%
  \ifx\svgwidth\undefined%
    \setlength{\unitlength}{122.7157205bp}%
    \ifx\svgscale\undefined%
      \relax%
    \else%
      \setlength{\unitlength}{\unitlength * \real{\svgscale}}%
    \fi%
  \else%
    \setlength{\unitlength}{\svgwidth}%
  \fi%
  \global\let\svgwidth\undefined%
  \global\let\svgscale\undefined%
  \makeatother%
  \begin{picture}(1,1.04026996)%
    \lineheight{1}%
    \setlength\tabcolsep{0pt}%
    \put(0,0){\includegraphics[width=\unitlength,page=1]{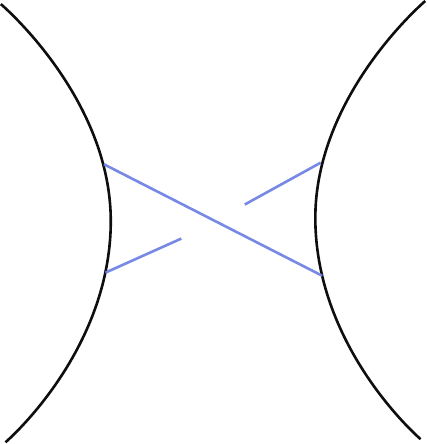}}%
  \end{picture}%
\endgroup%
}}
\begin{Lem}\label{lem:LES}
At a crossing of a link diagram, we have the following skein exact triangle in Khovanov homology, where the morphisms are induced by the saddles as indicated. (For simplicity we have suppressed all grading shifts.)
$$\begin{tikzcd}
\raisebox{9pt}{Kh\Bigg(}\scalebox{.2}{%% Creator: Inkscape 1.2.2 (732a01da63, 2022-12-09), www.inkscape.org
%% Accompanies image file 'crossing.pdf' (pdf, eps, ps)
\begingroup%
  \makeatletter%
  \providecommand\color[2][]{%
    \errmessage{(Inkscape) Color is used for the text in Inkscape, but the package 'color.sty' is not loaded}%
    \renewcommand\color[2][]{}%
  }%
  \providecommand\transparent[1]{%
    \errmessage{(Inkscape) Transparency is used (non-zero) for the text in Inkscape, but the package 'transparent.sty' is not loaded}%
    \renewcommand\transparent[1]{}%
  }%
  \providecommand\rotatebox[2]{#2}%
  \newcommand*\fsize{\dimexpr\f@size pt\relax}%
  \newcommand*\lineheight[1]{\fontsize{\fsize}{#1\fsize}\selectfont}%
  \ifx\svgwidth\undefined%
    \setlength{\unitlength}{127.6065826bp}%
    \ifx\svgscale\undefined%
      \relax%
    \else%
      \setlength{\unitlength}{\unitlength * \real{\svgscale}}%
    \fi%
  \else%
    \setlength{\unitlength}{\svgwidth}%
  \fi%
  \global\let\svgwidth\undefined%
  \global\let\svgscale\undefined%
  \makeatother%
  \begin{picture}(1,0.95477578)%
    \lineheight{1}%
    \setlength\tabcolsep{0pt}%
    \put(0,0){\includegraphics[width=\unitlength,page=1]{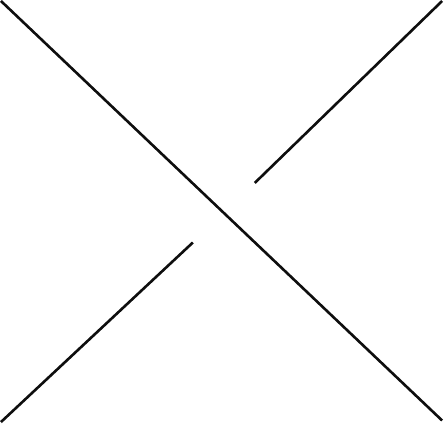}}%
  \end{picture}%
\endgroup%
}\raisebox{9pt}{\Bigg)}\ar[rr,"\usebox{\firstmor}",shift left=8pt]&&\raisebox{9pt}{Kh\Bigg(}\scalebox{.2}{%% Creator: Inkscape 1.2.2 (732a01da63, 2022-12-09), www.inkscape.org
%% Accompanies image file '0-resolution.pdf' (pdf, eps, ps)
\begingroup%
  \makeatletter%
  \providecommand\color[2][]{%
    \errmessage{(Inkscape) Color is used for the text in Inkscape, but the package 'color.sty' is not loaded}%
    \renewcommand\color[2][]{}%
  }%
  \providecommand\transparent[1]{%
    \errmessage{(Inkscape) Transparency is used (non-zero) for the text in Inkscape, but the package 'transparent.sty' is not loaded}%
    \renewcommand\transparent[1]{}%
  }%
  \providecommand\rotatebox[2]{#2}%
  \newcommand*\fsize{\dimexpr\f@size pt\relax}%
  \newcommand*\lineheight[1]{\fontsize{\fsize}{#1\fsize}\selectfont}%
  \ifx\svgwidth\undefined%
    \setlength{\unitlength}{127.65746639bp}%
    \ifx\svgscale\undefined%
      \relax%
    \else%
      \setlength{\unitlength}{\unitlength * \real{\svgscale}}%
    \fi%
  \else%
    \setlength{\unitlength}{\svgwidth}%
  \fi%
  \global\let\svgwidth\undefined%
  \global\let\svgscale\undefined%
  \makeatother%
  \begin{picture}(1,0.96128906)%
    \lineheight{1}%
    \setlength\tabcolsep{0pt}%
    \put(0,0){\includegraphics[width=\unitlength,page=1]{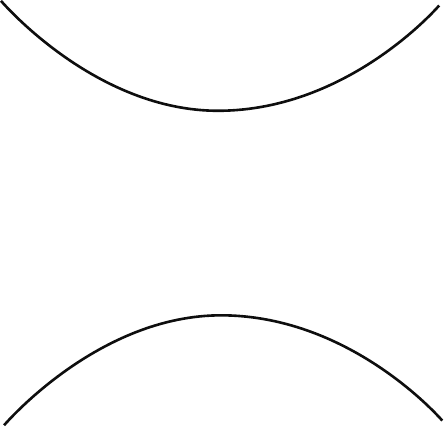}}%
  \end{picture}%
\endgroup%
}\raisebox{9pt}{\Bigg)}\ar[ld,"\usebox{\secondmor}"]\\
&\raisebox{9pt}{Kh\Bigg(}\scalebox{.2}{%% Creator: Inkscape 1.2.2 (732a01da63, 2022-12-09), www.inkscape.org
%% Accompanies image file '1-resolution.pdf' (pdf, eps, ps)
\begingroup%
  \makeatletter%
  \providecommand\color[2][]{%
    \errmessage{(Inkscape) Color is used for the text in Inkscape, but the package 'color.sty' is not loaded}%
    \renewcommand\color[2][]{}%
  }%
  \providecommand\transparent[1]{%
    \errmessage{(Inkscape) Transparency is used (non-zero) for the text in Inkscape, but the package 'transparent.sty' is not loaded}%
    \renewcommand\transparent[1]{}%
  }%
  \providecommand\rotatebox[2]{#2}%
  \newcommand*\fsize{\dimexpr\f@size pt\relax}%
  \newcommand*\lineheight[1]{\fontsize{\fsize}{#1\fsize}\selectfont}%
  \ifx\svgwidth\undefined%
    \setlength{\unitlength}{122.7157205bp}%
    \ifx\svgscale\undefined%
      \relax%
    \else%
      \setlength{\unitlength}{\unitlength * \real{\svgscale}}%
    \fi%
  \else%
    \setlength{\unitlength}{\svgwidth}%
  \fi%
  \global\let\svgwidth\undefined%
  \global\let\svgscale\undefined%
  \makeatother%
  \begin{picture}(1,1.04026996)%
    \lineheight{1}%
    \setlength\tabcolsep{0pt}%
    \put(0,0){\includegraphics[width=\unitlength,page=1]{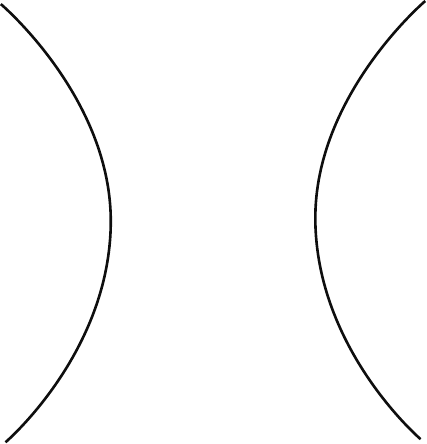}}%
  \end{picture}%
\endgroup%
}\raisebox{9pt}{\Bigg)}\ar[lu,"\usebox{\thirdmor}"]&
\end{tikzcd}$$
\end{Lem}
\begin{pf}
Let $CKh$ denote the Khovanov chain complex. Then by definition $$CKh(\,\raisebox{-3pt}{\scalebox{.10}{}}\,)=Cone(CKh(\,\raisebox{-3pt}{\scalebox{.10}{}}\,)\xrightarrow{\scalebox{.5}{{\usebox{\secondmor}}}}CKh(\,\raisebox{-3pt}{\scalebox{.10}{}}\,)).$$ This gives rise to an exact triangle as stated, except that we still have to check the top and the left morphisms in the exact triangle agree with the morphisms induced by the saddle cobordisms given in the statement.

First we check the top morphism. Identify $CKh(\,\raisebox{-3pt}{\scalebox{.10}{}}\,)=CKh(\,\raisebox{-3pt}{\scalebox{.10}{}}\,)\oplus CKh(\,\raisebox{-3pt}{\scalebox{.10}{}}\,)$ as modules, the top morphism in the exact triangle is induced by the projection of $CKh(\,\raisebox{-3pt}{\scalebox{.10}{}}\,)$ onto $CKh(\,\raisebox{-3pt}{\scalebox{.10}{}}\,)$. On the other hand, the morphism given in the statement is induced by the chain map 
\begin{equation}\label{eq:crossing_to}
CKh(\,\raisebox{-3pt}{\scalebox{.10}{}}\,)\xrightarrow{\scalebox{.10}{}}CKh(\,\raisebox{-3pt}{\scalebox{.10}{%% Creator: Inkscape 1.2.2 (732a01da63, 2022-12-09), www.inkscape.org
%% Accompanies image file 'crossing_to_2.pdf' (pdf, eps, ps)
\begingroup%
  \makeatletter%
  \providecommand\color[2][]{%
    \errmessage{(Inkscape) Color is used for the text in Inkscape, but the package 'color.sty' is not loaded}%
    \renewcommand\color[2][]{}%
  }%
  \providecommand\transparent[1]{%
    \errmessage{(Inkscape) Transparency is used (non-zero) for the text in Inkscape, but the package 'transparent.sty' is not loaded}%
    \renewcommand\transparent[1]{}%
  }%
  \providecommand\rotatebox[2]{#2}%
  \newcommand*\fsize{\dimexpr\f@size pt\relax}%
  \newcommand*\lineheight[1]{\fontsize{\fsize}{#1\fsize}\selectfont}%
  \ifx\svgwidth\undefined%
    \setlength{\unitlength}{127.63103449bp}%
    \ifx\svgscale\undefined%
      \relax%
    \else%
      \setlength{\unitlength}{\unitlength * \real{\svgscale}}%
    \fi%
  \else%
    \setlength{\unitlength}{\svgwidth}%
  \fi%
  \global\let\svgwidth\undefined%
  \global\let\svgscale\undefined%
  \makeatother%
  \begin{picture}(1,0.95515725)%
    \lineheight{1}%
    \setlength\tabcolsep{0pt}%
    \put(0,0){\includegraphics[width=\unitlength,page=1]{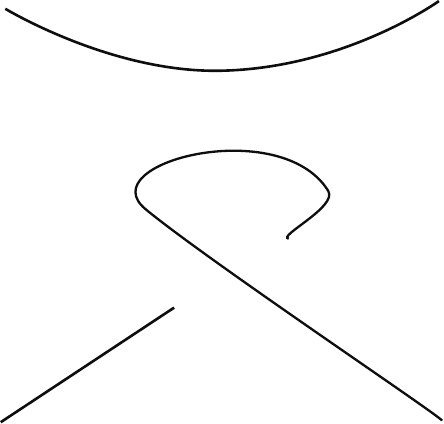}}%
  \end{picture}%
\endgroup%
}}\,)\xrightarrow{(R1^+)^{-1}}CKh(\,\raisebox{-3pt}{\scalebox{.10}{}}\,)
\end{equation}
where $(R1^+)^{-1}$ denotes the chain homotopy equivalence induced by undoing the positive twist. See \cite[Table~1,Table~3]{hayden2021khovanov} for succinct descriptions of the induced maps by saddle cobordisms and Reidemeister I moves. On the direct summand $CKh(\,\raisebox{-3pt}{\scalebox{.10}{}}\,)$, \eqref{eq:crossing_to} equals
\begin{equation}\label{eq:crossing_to1}
CKh(\,\raisebox{-3pt}{\scalebox{.10}{}}\,)\to CKh(\,\raisebox{-3pt}{\scalebox{.10}{%% Creator: Inkscape 1.2.2 (732a01da63, 2022-12-09), www.inkscape.org
%% Accompanies image file 'crossing_to_2_1.pdf' (pdf, eps, ps)
\begingroup%
  \makeatletter%
  \providecommand\color[2][]{%
    \errmessage{(Inkscape) Color is used for the text in Inkscape, but the package 'color.sty' is not loaded}%
    \renewcommand\color[2][]{}%
  }%
  \providecommand\transparent[1]{%
    \errmessage{(Inkscape) Transparency is used (non-zero) for the text in Inkscape, but the package 'transparent.sty' is not loaded}%
    \renewcommand\transparent[1]{}%
  }%
  \providecommand\rotatebox[2]{#2}%
  \newcommand*\fsize{\dimexpr\f@size pt\relax}%
  \newcommand*\lineheight[1]{\fontsize{\fsize}{#1\fsize}\selectfont}%
  \ifx\svgwidth\undefined%
    \setlength{\unitlength}{125.29086848bp}%
    \ifx\svgscale\undefined%
      \relax%
    \else%
      \setlength{\unitlength}{\unitlength * \real{\svgscale}}%
    \fi%
  \else%
    \setlength{\unitlength}{\svgwidth}%
  \fi%
  \global\let\svgwidth\undefined%
  \global\let\svgscale\undefined%
  \makeatother%
  \begin{picture}(1,1.0090206)%
    \lineheight{1}%
    \setlength\tabcolsep{0pt}%
    \put(0,0){\includegraphics[width=\unitlength,page=1]{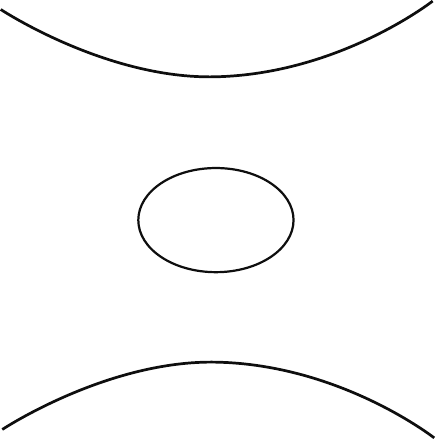}}%
  \end{picture}%
\endgroup%
}}\,)\to CKh(\,\raisebox{-3pt}{\scalebox{.10}{}}\,)
\end{equation}
where the first map is induced by the splitting saddle cobordism $\Delta\colon1\mapsto1\otimes X+X\otimes1,X\mapsto X\otimes X$ on the top strand, and the second map is induced by the death cobordism $\epsilon\colon1\mapsto0,X\mapsto1$ on the middle circle. Since $(1\otimes\epsilon)\circ\Delta=1$, the composition \eqref{eq:crossing_to1} is the identity. On the direct summand $CKh(\,\raisebox{-3pt}{\scalebox{.10}{}}\,)$, \eqref{eq:crossing_to} equals $$CKh(\,\raisebox{-3pt}{\scalebox{.10}{}}\,)\to CKh(\,\raisebox{-3pt}{\scalebox{.10}{%% Creator: Inkscape 1.2.2 (732a01da63, 2022-12-09), www.inkscape.org
%% Accompanies image file 'crossing_to_2_2.pdf' (pdf, eps, ps)
\begingroup%
  \makeatletter%
  \providecommand\color[2][]{%
    \errmessage{(Inkscape) Color is used for the text in Inkscape, but the package 'color.sty' is not loaded}%
    \renewcommand\color[2][]{}%
  }%
  \providecommand\transparent[1]{%
    \errmessage{(Inkscape) Transparency is used (non-zero) for the text in Inkscape, but the package 'transparent.sty' is not loaded}%
    \renewcommand\transparent[1]{}%
  }%
  \providecommand\rotatebox[2]{#2}%
  \newcommand*\fsize{\dimexpr\f@size pt\relax}%
  \newcommand*\lineheight[1]{\fontsize{\fsize}{#1\fsize}\selectfont}%
  \ifx\svgwidth\undefined%
    \setlength{\unitlength}{125.16979069bp}%
    \ifx\svgscale\undefined%
      \relax%
    \else%
      \setlength{\unitlength}{\unitlength * \real{\svgscale}}%
    \fi%
  \else%
    \setlength{\unitlength}{\svgwidth}%
  \fi%
  \global\let\svgwidth\undefined%
  \global\let\svgscale\undefined%
  \makeatother%
  \begin{picture}(1,1.00983831)%
    \lineheight{1}%
    \setlength\tabcolsep{0pt}%
    \put(0,0){\includegraphics[width=\unitlength,page=1]{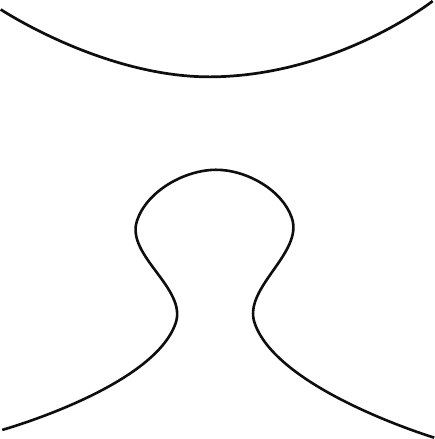}}%
  \end{picture}%
\endgroup%
}}\,)\to CKh(\,\raisebox{-3pt}{\scalebox{.10}{}}\,),$$ where the second map is identically zero. We have thus shown that the top morphism in the exact triangle equals the stated morphism.

Next we check the left cobordism. The morphism in the exact triangle is induced by the inclusion of $CKh(\,\raisebox{-3pt}{\scalebox{.10}{}}\,)$ into $CKh(\,\raisebox{-3pt}{\scalebox{.10}{}}\,)$ as a direct summand. On the other hand, the morphism given in the statement is induced by the chain map
\begin{equation}\label{eq:1-resolution_to}
CKh(\,\raisebox{-3pt}{\scalebox{.10}{}}\,)\xrightarrow{R1^-}CKh(\,\raisebox{-3pt}{\scalebox{.10}{%% Creator: Inkscape 1.2.2 (732a01da63, 2022-12-09), www.inkscape.org
%% Accompanies image file '1-resolution_to_2.pdf' (pdf, eps, ps)
\begingroup%
  \makeatletter%
  \providecommand\color[2][]{%
    \errmessage{(Inkscape) Color is used for the text in Inkscape, but the package 'color.sty' is not loaded}%
    \renewcommand\color[2][]{}%
  }%
  \providecommand\transparent[1]{%
    \errmessage{(Inkscape) Transparency is used (non-zero) for the text in Inkscape, but the package 'transparent.sty' is not loaded}%
    \renewcommand\transparent[1]{}%
  }%
  \providecommand\rotatebox[2]{#2}%
  \newcommand*\fsize{\dimexpr\f@size pt\relax}%
  \newcommand*\lineheight[1]{\fontsize{\fsize}{#1\fsize}\selectfont}%
  \ifx\svgwidth\undefined%
    \setlength{\unitlength}{121.9077073bp}%
    \ifx\svgscale\undefined%
      \relax%
    \else%
      \setlength{\unitlength}{\unitlength * \real{\svgscale}}%
    \fi%
  \else%
    \setlength{\unitlength}{\svgwidth}%
  \fi%
  \global\let\svgwidth\undefined%
  \global\let\svgscale\undefined%
  \makeatother%
  \begin{picture}(1,1.04694803)%
    \lineheight{1}%
    \setlength\tabcolsep{0pt}%
    \put(0,0){\includegraphics[width=\unitlength,page=1]{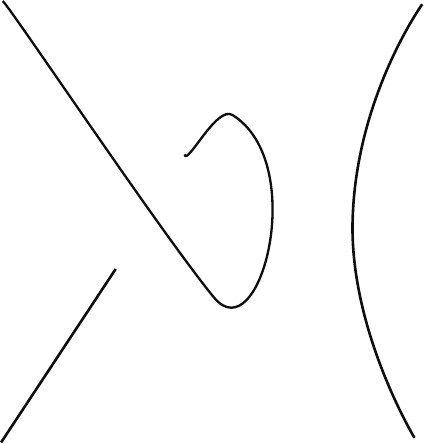}}%
  \end{picture}%
\endgroup%
}}\,)\xrightarrow{\scalebox{.10}{%% Creator: Inkscape 1.2.2 (732a01da63, 2022-12-09), www.inkscape.org
%% Accompanies image file '1-resolution_to_2_to.pdf' (pdf, eps, ps)
\begingroup%
  \makeatletter%
  \providecommand\color[2][]{%
    \errmessage{(Inkscape) Color is used for the text in Inkscape, but the package 'color.sty' is not loaded}%
    \renewcommand\color[2][]{}%
  }%
  \providecommand\transparent[1]{%
    \errmessage{(Inkscape) Transparency is used (non-zero) for the text in Inkscape, but the package 'transparent.sty' is not loaded}%
    \renewcommand\transparent[1]{}%
  }%
  \providecommand\rotatebox[2]{#2}%
  \newcommand*\fsize{\dimexpr\f@size pt\relax}%
  \newcommand*\lineheight[1]{\fontsize{\fsize}{#1\fsize}\selectfont}%
  \ifx\svgwidth\undefined%
    \setlength{\unitlength}{121.89448057bp}%
    \ifx\svgscale\undefined%
      \relax%
    \else%
      \setlength{\unitlength}{\unitlength * \real{\svgscale}}%
    \fi%
  \else%
    \setlength{\unitlength}{\svgwidth}%
  \fi%
  \global\let\svgwidth\undefined%
  \global\let\svgscale\undefined%
  \makeatother%
  \begin{picture}(1,1.04737818)%
    \lineheight{1}%
    \setlength\tabcolsep{0pt}%
    \put(0,0){\includegraphics[width=\unitlength,page=1]{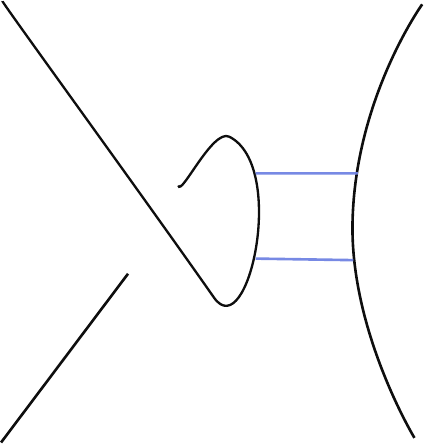}}%
  \end{picture}%
\endgroup%
}}CKh(\,\raisebox{-3pt}{\scalebox{.10}{}}\,)
\end{equation}
where $R1^-$ denotes the chain homotopy equivalence induced by creating the negative twist, which is equal to the map $CKh(\,\raisebox{-3pt}{\scalebox{.10}{}}\,)\to CKh(\,\raisebox{-3pt}{\scalebox{.10}{%% Creator: Inkscape 1.2.2 (732a01da63, 2022-12-09), www.inkscape.org
%% Accompanies image file '1-resolution_to_2_1.pdf' (pdf, eps, ps)
\begingroup%
  \makeatletter%
  \providecommand\color[2][]{%
    \errmessage{(Inkscape) Color is used for the text in Inkscape, but the package 'color.sty' is not loaded}%
    \renewcommand\color[2][]{}%
  }%
  \providecommand\transparent[1]{%
    \errmessage{(Inkscape) Transparency is used (non-zero) for the text in Inkscape, but the package 'transparent.sty' is not loaded}%
    \renewcommand\transparent[1]{}%
  }%
  \providecommand\rotatebox[2]{#2}%
  \newcommand*\fsize{\dimexpr\f@size pt\relax}%
  \newcommand*\lineheight[1]{\fontsize{\fsize}{#1\fsize}\selectfont}%
  \ifx\svgwidth\undefined%
    \setlength{\unitlength}{126.42106687bp}%
    \ifx\svgscale\undefined%
      \relax%
    \else%
      \setlength{\unitlength}{\unitlength * \real{\svgscale}}%
    \fi%
  \else%
    \setlength{\unitlength}{\svgwidth}%
  \fi%
  \global\let\svgwidth\undefined%
  \global\let\svgscale\undefined%
  \makeatother%
  \begin{picture}(1,0.99106005)%
    \lineheight{1}%
    \setlength\tabcolsep{0pt}%
    \put(0,0){\includegraphics[width=\unitlength,page=1]{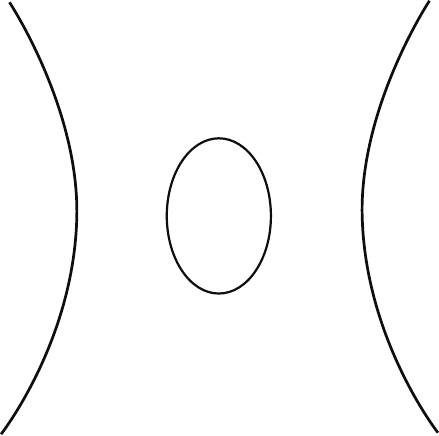}}%
  \end{picture}%
\endgroup%
}}\,)$ induced by the birth cobordism $\iota\colon1\mapsto1$ creating the middle circle, followed by the inclusion $CKh(\,\raisebox{-3pt}{\scalebox{.10}{}}\,)\to CKh(\,\raisebox{-3pt}{\scalebox{.10}{}}\,)$ as a direct summand. On the direct summand $CKh(\,\raisebox{-3pt}{\scalebox{.10}{}}\,)$, the second map in \eqref{eq:1-resolution_to} equals the merging saddle cobordism $m\colon1\otimes1\mapsto1,1\otimes X\mapsto X,X\otimes1\mapsto X,X\otimes X\mapsto0$ on the right two components. Since $m\circ(\iota\otimes1)=1$, the composition \eqref{eq:1-resolution_to} equals the inclusion map, and the proof is complete.
\end{pf}

Next we introduce some notations that will be convenient. For two functions $f,g\colon\Z\to\Z\sqcup\{+\infty\}$, define $\min(f,g)$ to be the function $\Z\to\Z\sqcup\{+\infty\}$ with $\min(f,g)(h)=\min(f(h),g(h))$. We write $f>g$ if $f(h)>g(h)$ for all $h$ with $f(h)<+\infty$, and write $f<g$, $f\ge g$, $f\le g$ analogously. For $h,q\in\Z$, we define $f[h]\{q\}\colon\Z\to\Z\sqcup\{+\infty\}$ by $(f[h]\{q\})(h')=f(h'-h)+q$. For a function defined on a subset of $\Z$ that values in $\Z\sqcup\{+\infty\}$, we abuse the notation and use the same expression to denote its extension by $+\infty$ to all of $\Z$. Finally, define $t_{n,n}(h):=\inf\{q\colon Kh^{h,q}(T_{n,n})\ne0\}$ to be the quantum infimum function for $Kh(T_{n,n})$, and similarly $t_{n+1,n},d_{n,m}^i,e_{n,m}^i$ to be the quantum infimum functions for $Kh(T_{n+1,n})$, $Kh(D_{n,m}^i)$, $Kh(E_{n,m}^i)$, respectively.

Thus, for example, the lower bound of Theorem~\ref{thm:lower_bound}(1) says $t_{n,n}\ge q_{n,n}$.

Before launching into the proof of Theorem~\ref{thm:lower_bound}, we note the following relations among the functions $q_{n,n}$ and $q_{n+1,n}$ defined in Section~\ref{sec:s_bound}.

\begin{Lem}\label{lem:q_relations}
\begin{subequations}
\begin{gather}
q_{n-2,n-2}[2n-2]\{6n-8\}=q_{n,n}|_{h\ge2n-2}\ge q_{n,n},\label{eq:qn_n}\\
q_{n-1,n-2}[2n-2]\{6n-6\}\ge q_{n+1,n}|_{h\le h_{max}(T_{n+1,n})-1},\label{eq:qn-1_n-1}\\
q_{n,n}\{n-1\}\ge q_{n+1,n},\label{eq:qn_n-1}\\
q_{n,n-1}\{n-1\}\ge q_{n,n},\label{eq:qn-1_n}\\
q_{n,n-1}[1]\{2\}|_{1\ne h\le h_{max}(T_{n,n-1})}\ge q_{n,n-1},\label{eq:12}\\
q_{n,n-1}[2]\{4\}|_{h\le h_{max}(T_{n,n-1})}\ge q_{n,n-1}.\label{eq:24}
\end{gather}
\end{subequations}
Moreover, \eqref{eq:qn_n-1} is strict at $h=2n-1$. \eqref{eq:qn-1_n} is strict at $h=2pq$ for any $p+q=n$, $p\ge q>0$; it is strict at $h=2pq-1$ for any $p+q=n$, $p\ge q>1$.
\end{Lem}\vspace{-6pt}

These are elementary. We sketch mostly geometric proofs to them.\vspace{-6pt}
\begin{proof}[Sketch proof]
The functions $q_{n,n}$ are thought of as staircases as indicated by the red segments in Figure~\ref{fig:T66}. The height of a step is usually $2$, except $4$ right after homological degrees $2pq$. Thus to check the equality in \eqref{eq:qn_n}, it suffices to note that the two sets $\{2pq\colon p+q=n,p\ge q>0\}$ and $\{2pq\colon p+q=n-2,p\ge q\ge0\}$ are identical up to a shift of $2n-2$, and that $q_{n,n}(2n-2)=q_{n-2,n-2}(0)+6n-8$.

The function $q_{n+1,n}$ is thought of as a shift of $q_{n,n}$ by $\{n-1\}$ with all but the first big step flatten (by decreasing the height by $2$ at $2pq+1$, $p\ge q>0$), plus $\lfloor n/2\rfloor$ short small steps as tail, see Figure~\ref{fig:T76}. This description immediately gives \eqref{eq:qn_n-1}\eqref{eq:12}\eqref{eq:24}. Together with \eqref{eq:qn_n}, this also proves \eqref{eq:qn-1_n-1} (we can apply the truncation in the end because the tail of $q_{n+1,n}$ is one longer than that of $q_{n-1,n-2}$). Since $2n-1=2(n-1)\cdot1+1$, \eqref{eq:qn_n-1} is strict at $h=2n-1$.

Finally, \eqref{eq:qn-1_n} is more contrived, so we first calculate algebraically. At $h=0$, \eqref{eq:qn-1_n} is an equality. For $0<h\le h_{max}(T_{n-1,n-1})$, choose $p+q=n,p\ge q>0$ and $p'+q'=n-1,p'\ge q'>0$ with $h\in(2(p+1)(q-1),2pq]\cap(2(p'+1)(q'-1),2p'q']$. Then
\begin{equation}\label{eq:q_adj_compare}
q_{n-1,n-1}(h)+2n-3-q_{n,n}(h)=(n-1)^2+2\left\lceil\frac h2\right\rceil-2p'+2n-3-n^2-2\left\lceil\frac h2\right\rceil+2p=2(q'-q).
\end{equation}
If $h\ne 2(p'+1)(q'-1)+1$ or $q'=1$, \eqref{eq:q_adj_compare} gives $q_{n,n-1}\{n-1\}(h)=q_{n-1,n-1}(h)+2n-3\ge q_{n,n}(h)$ since $q'\ge q$. If $h=2(p'+1)(q'-1)+1$ and $q'\ne1$, \eqref{eq:q_adj_compare} gives $q_{n,n-1}\{n-1\}(h)=q_{n-1,n-1}(h)+2n-5=q_{n,n}(h)$ since $q'>q$.

To prove \eqref{eq:qn-1_n} for $h_{max}(T_{n-1,n-1})\le h\le h_{max}(T_{n,n-1})$, we note that in this range $q_{n,n-1}\{n-1\}$ is a tail of short small steps while $q_{n,n}$ is a combination of long small and long big steps. Thus \eqref{eq:qn-1_n} follows from its validity at $h=h_{max}(T_{n-1,n-1})$, unless $q_{n,n}$ has a big step right after $h=h_{max}(T_{n-1,n-1})$. In this exceptional case, $h_{max}(T_{n-1,n-1})=2pq$ for some $p+q=n$, $p\ge q>0$. This implies $q'>q$ in \eqref{eq:q_adj_compare}, so \eqref{eq:qn-1_n} is strict at $h=h_{max}(T_{n-1,n-1})$, and \eqref{eq:qn-1_n} also follows.

Finally we prove the addendum about the strictness of \eqref{eq:qn-1_n}. For $0<h\le h_{max}(T_{n-1,n-1})$, $h\in(2(p+1)(q-1),2pq]\cap(2(p'+1)(q'-1),2p'q']$, \eqref{eq:qn-1_n} is strict if and only if $q'>q+1$ or $q'=q+1$ and $h\ne2(p'+1)(q'-1)+1$. This is the case for $h=2pq$ if $p\ge q>0$, and for $h=2pq-1$ if $p\ge q>1$. 

Now assume $h=2pq,2pq-1$, $h>h_{max}(T_{n-1,n-1})$. The case for $h=2pq$ is trivial because $$q_{n,n-1}\{n-1\}(2pq)=q_{n,n-1}\{n-1\}(2pq-1)+2\ge q_{n,n}(2pq-1)+2=q_{n,n-1}(2pq)+2.$$ For $h=2pq-1$, we divide into three cases.

\textit{Case 1}: If $2pq-1>\max(2(p+1)(q-1)+1,h_{max}(T_{n-1,n-1})+1)$, then $$q_{n,n-1}\{n-1\}(2pq-1)=q_{n,n-1}\{n-1\}(2pq-3)+4\ge q_{n,n}(2pq-3)+4=q_{n,n}(2pq-1)+2.$$
\textit{Case 2}: If $2pq-1=2(p+1)(q-1)+1$, then $p=q+1$, so $h_{max}(T_{n,n-1})=2q^2+q$ is less than $2pq-1$ unless $q=1$, which is excluded in our hypothesis.\smallskip\\
\textit{Case 3}: If $2pq-1=h_{max}(T_{n-1,n-1})+1>2(p+1)(q-1)+1$, then the strictness of \eqref{eq:qn-1_n} at $2pq-1$ is equivalent to that at $h_{max}(T_{n-1,n-1})$. The latter is true because $q'>q$ in \eqref{eq:q_adj_compare} unless $q=1$.
\end{proof}

\begin{proof}[Proof of Theorem~\ref{thm:lower_bound}]
In the notations above, we need to prove $t_{n,n}\ge q_{n,n}$, $t_{n+1,n}\ge q_{n+1,n}$, and the two addenda to (1)(2).

We induct on $n$. The base cases $n=1,2$ are easily checked. Below we assume $n\ge3$.

(1) The statements for $T_{n,n}$.

The long exact sequence \eqref{eq:LES_n} gives
\begin{equation}\label{eq:induct_Dn}
d_{n,n-1}^i\ge\min(e_{n,n-1}^{i-1}[2n-2]\{6n-7\},d_{n,n-1}^{i-1}\{1\}).
\end{equation}
Using Lemma~\ref{lem:E}, we thus have $$t_{n,n}\ge\min(t_{n-2,n-2}[2n-2]\{6n-8\},d_{n,n-1}^{n-2}\{1\})=:\min(A,B),$$ where $A\ge q_{n-2,n-2}[2n-2]\{6n-8\}\ge q_{n,n}$ by induction hypothesis and \eqref{eq:qn_n}. Inductively applying \eqref{eq:induct_Dn} (cf. Figure~\ref{fig:induct}) we obtain
\begin{figure}
\centering
\scalebox{.75}{\input{DE_induct}}
\caption{Flowchart of applying \eqref{eq:induct_Dn} to $D_{6,5}^4$}
\label{fig:induct}
\end{figure}
\begin{align}\label{eq:B}
B\ge&\,\min\bigg(\min_{\substack{n-m=2r\ge0\\m\ge2}}t_{m,m-1}\left[\sum\nolimits'(2i-2)\right]\left\{\sum\nolimits'(6i-8)+n-1\right\},\nonumber\\
&\,\min_{\substack{n-m=2r>0\\m\ge1}}t_{m,m}\left[\sum\nolimits'(2i-2)\right]\left\{\sum\nolimits'(6i-8)+n-m\right\}\bigg).
\end{align}
Here and henceforth, each $\Sigma'$ is a sum over $i\in(m,n]$ with the same parity as $m,n$. By induction hypothesis and \eqref{eq:qn_n}, each term in the second sum in \eqref{eq:B} is bounded below by $$q_{m,m}\left[\sum\nolimits'(2i-2)\right]\left\{\sum\nolimits'(6i-8)+n-m\right\}\ge q_{n,n}\{n-m\}>q_{n,n}.$$
Similarly, by induction hypothesis and \eqref{eq:qn-1_n-1}\eqref{eq:qn-1_n}\eqref{eq:24}, each term in the first sum in \eqref{eq:B} is bounded below by
\begin{align*}
&q_{m,m-1}\left[\sum\nolimits'(2i-4)+2r\right]\left\{\sum\nolimits'(6i-12)+4r+n-1\right\}\\
\ge&\,q_{n,n-1}[2r]\{4r+n-1\}\ge\min(q_{n,n-1}\{n-1\},tail)\ge q_{n,n}.
\end{align*}
where $tail$ is the function defined on $[h_{max}(T_{n,n-1}),h_{max}(T_{n,n-1})+2r]$ which consists of short small steps (in the sense of the proof of Lemma~\ref{lem:q_relations}) and agrees with $q_{n,n-1}\{n-1\}$ at $h=h_{max}(T_{n,n-1})$.

Now it remains to prove the statement about the saddle cobordism. The relevant induced map $\alpha$ fits in the long exact sequence \eqref{eq:LES_n} 
\begin{align}\label{eq:LES_piece}
\cdots\to Kh^{h-1,q-1}(D_{n,n-1}^{n-2})\xrightarrow{\beta}&Kh^{h-2n+2,q-6n+7}(T_{n-2,n-2}\sqcup U)\nonumber\\
&\xrightarrow{\alpha}Kh^{h,q}(T_{n,n})\to Kh^{h,q-1}(D_{n,n-1}^{n-2})\to\cdots
\end{align}
for $h=2pq>0$, $q=q_{n,n}(h)$. The map $\alpha$ would be an isomorphism if $Kh^{h-1,q-1}(D_{n,n-1}^{n-2})=Kh^{h,q-1}(D_{n-1,n-2}^{n-2})=0$, or equivalently, using the notation above, that $B(h-1),B(h)>q(=q_{n,n}(h)=q_{n,n}(h-1))$.

We re-examine the estimate $B\ge q_{n,n}$ above. The contribution from the second term in \eqref{eq:B} is always strict. The contribution from the first term is strict for $h$, and is strict for $h-1$ if $q>1$, because \eqref{eq:qn-1_n} is strict for such homological degrees by the second addendum of Lemma~\ref{lem:q_relations}. If $q=1$, however, the inequality is strict at $h-1=2n-3$ for $r>0$ (as the left hand side is $+\infty$), but not for $r=0$. In this exceptional case, the addendum in the induction hypothesis for $T_{n,n-1}$ states that the relevant homology group $Kh^{2n-3,q_{n,n-1}(2n-3)}(T_{n,n-1})$ is torsion, thus so is every $Kh^{2n-3,q_{n,n-1}(2n-3)+i}(D_{n,n-1}^i)$, in view of \eqref{eq:LES_n-1}. This group for $i=n-2$ and the group $Kh^{0,(n-3)^2-2}(T_{n-2,n-2}\sqcup U)=\Z$ appear as the first two terms in \eqref{eq:LES_piece}, which implies $\beta$ in \eqref{eq:LES_piece} is zero, thus $\alpha$ is an isomorphism.\medskip

(2) The statements for $T_{n+1,n}$.

The long exact sequence \eqref{eq:LES_n-1} gives
\begin{equation}\label{eq:induct_Dn-1}
d_{n,n}^i\ge\min(e_{n,n}^{i-1}[2n-1]\{6n-4\},d_{n,n}^{i-1}\{1\}).
\end{equation}
Using Lemma~\ref{lem:E}, we obtain $$t_{n+1,n}\ge\min(t_{n-1,n-2}[2n-1]\{6n-4\},d_{n,n}^{n-2}\{1\})=:\min(A,B)$$ where $A\ge q_{n-1,n-2}[2n-1]\{6n-4\}\ge q_{n+1,n}[1]\{2\}|_{2n-1\le h\le h_{max}(T_{n+1,n})}\ge q_{n+1,n}$ by induction hypothesis and \eqref{eq:qn-1_n-1}\eqref{eq:12}, and 
\begin{equation}\label{eq:B2}
B\ge\min_{n-m=2r\ge0}t_{m,m}\left[\sum\nolimits'(2i-1)\right]\left\{\sum\nolimits'(6i-6)+n-1\right\}.
\end{equation}
By induction hypothesis and \eqref{eq:qn_n}\eqref{eq:qn_n-1}\eqref{eq:12}, the $r=0$ term in the summation is bounded below by $q_{n,n}\{n-1\}\ge q_{n+1,n}$, and each of the $r>0$ terms is bounded below by 
\begin{align*}
&q_{m,m}\left[\sum\nolimits'(2i-1)\right]\left\{\sum\nolimits'(6i-6)+n-1\right\}\ge q_{n,n}[r]\{2r+n-1\}|_{h\ge2n-1}\\
\ge&\,q_{n+1,n}[r]\{2r\}|_{2n-1\le h\le h_{max}(T_{n,n})+r}\ge q_{n+1,n}.
\end{align*}

It remains to show $Kh^{2n-1,q_{n+1,n}(2n-1)}(T_{n+1,n})$ is torsion. This group sits in the long exact sequence \eqref{eq:LES_n-1}
\begin{align*}
\cdots\to Kh^{2n-2,q_{n+1,n}(2n-1)-1}&(D_{n,n}^{n-2})\xrightarrow{\gamma}Kh^{0,q_{n-1,n-2}(0)}(T_{n-1,n-2})\\\to& Kh^{2n-1,q_{n+1,n}(2n-1)}(T_{n+1,n})\to Kh^{2n-1,q_{n+1,n}(2n-1)-1}(D_{n,n}^{n-2})\to\cdots.
\end{align*}
Thus it suffices to show, upon tensoring $\Q$, that $\gamma$ is surjective and $Kh^{2n-1,q_{n+1,n}(2n-1)-1}(D_{n,n}^{n-2})=0$.

By Lemma~\ref{lem:LES}, $\gamma$ is induced by a saddle cobordism $D_{n,n}^{n-2}\to T_{n-1,n-2}$. Let $\theta\colon Kh^{0,q_{n-1,n-2}(0)+2}($ $T_{n-1,n-2})\to Kh^{2n-2,q_{n+1,n}(2n-1)-1}(D_{n,n}^{n-2})$ be the map induced by the backward saddle cobordism. Then by the neck-cutting relation \cite[Section~11.2]{bar2005khovanov} we have $\gamma\circ\theta=2X$, where 
\begin{equation}\label{eq:dotted}
X\colon Kh^{0,q_{n-1,n-2}(0)+2}(T_{n-1,n-2})\to Kh^{0,q_{n-1,n-2}(0)}(T_{n-1,n-2})
\end{equation}
is the map induced by the dotted cobordism on $T_{n-1,n-2}$. By Sto\v{s}i\'c \cite[Theorem~3.4]{stovsic2007homological}, $Kh^{0,*}(T_{n-1,n-2})=\Z$ for $*=q_{n-1,n-2}(0)+1\pm1$ and $0$ otherwise, and $Kh^{1,*}(T_{n-1,n-2})=0$. Hence the reduced Khovanov homology of $T_{n-1,n-2}$ has $\widetilde{Kh}^{1,*}(T_{n-1,n-2})=0$, and thus $\widetilde{Kh}^{0,*}(T_{n-1,n-2})=\Z$ for $*=q_{n-1,n-2}(0)+1$ and $0$ otherwise, in view of the long exact sequence $$\cdots\to\widetilde{Kh}^{-1,*-1}\to\widetilde{Kh}^{0,*+1}\to Kh^{0,*}\to\widetilde{Kh}^{0,*-1}\to\widetilde{Kh}^{1,*+1}\to\cdots$$ applied to $T_{n-1,n-2}$. Consequently, \eqref{eq:dotted} is an isomorphism. This proves the surjectivity of $\gamma\otimes\Q$.

Next we show $Kh^{2n-1,q_{n+1,n}(2n-1)-1}(D_{n,n}^{n-2})\otimes\Q=0$. Let $d_{n,m,\Q}^i$ denote the quantum infimum function for $Kh(D_{n,m}^i)\otimes\Q$. The above estimates apply equally well with $d_{n,m}^i$ replaced by $d_{n,m,\Q}^i$. We now need to show the inequality $d_{n,n,\Q}^{n-2}\{1\}\ge q_{n+1,n}$ is strict at homological degree $2n-1$. We re-examine the estimate for $B$ above. The contribution from the $r=0$ term in \eqref{eq:B2} is strict at $2n-1$, because \eqref{eq:qn_n-1} is strict at $2n-1$ by the first addendum of Lemma~\ref{lem:q_relations}. The $r\ge2$ terms are also strict because the left hand sides are $+\infty$. For the $r=1$ term, the contribution is not necessarily strict, but it would be if we could improve the contribution coming from the first term in \eqref{eq:induct_Dn-1} at homological degree $2n-1$ by an extra positive quantum shift, for each $i$. Upon tensoring $\Q$, we can indeed make this improvement, by showing that the relevant maps $Kh^{2n-2,q_{n-2,n-2}(0)+i-2+6n-5}(D_{n,n}^{i-1})\otimes\Q\to Kh^{0,q_{n-2,n-2}(0)+i-2}(E_{n,n}^{i-1})\otimes\Q$ in \eqref{eq:LES_n-1}$\otimes\Q$ are surjective. Note that for $i=n-1$ this map is exactly $\gamma\otimes\Q$, whose surjectivity has just been shown. The $i<n-1$ cases follow from exactly the same argument, using the fact that, by an induction on $i$ using \eqref{eq:LES_n-1}, $Kh^{0,*}(D_{n-2,n-2}^i)\otimes\Q=\Q$ for $*=q_{n-2,n-2}(0)+i+1\pm1$ and $0$ otherwise, and $Kh^{1,*}(D_{n-2,n-2}^i)=0$ (the base case for $D_{n-2,n-2}^0=T_{n-2,n-2}$ follows again from \cite[Theorem~3.4]{stovsic2007homological}).
\end{proof}

\begin{Rmk}
The argument above carries through with $\Q$ replaced by any $\mathbb{F}_p$, $p\ne2$. This shows the order of every element in $Kh^{2n-1,q_{n+1,n}(2n-1)}(T_{n+1,n})$ is a power of $2$.
\end{Rmk}

\section{Questions}\label{sec:questions}
In this section we make some comments mostly related to Theorem~\ref{thm:lower_bound}. First, we remark that the statements in Theorem~\ref{thm:lower_bound} are almost designed minimally so that the $m=n$ case of Theorem~\ref{thm:s} can be proven. One may try to prove more about $Kh(T(n,n))$ and $Kh(T(n+1,n))$ of their own interests using the same induction scheme. For example, it seems true that the bound $q_{n,n}$ is sharp not only for $h=2pq$, but for any even $h$; not only $Kh^{2n-1,q_{n+1,n}(2n-1)}(T(n+1,n))$ but also any $Kh^{2pq+1,q_{n+1,n}(2pq+1)}(T(n+1,n))$ is torsion, which equals $0$ if $n$ is odd. Moreover, one can similarly expect to obtain a ``graphical upper bound'' for $Kh(T(n,n))$ and $Kh(T(n+1,n))$. This might lead to a proof of Corollary~\ref{cor:bar_s} (over any coefficient field) without resorting to Theorem~\ref{thm:filtration}. One may also try to prove similar results on Khovanov homology of more general torus links $T(n,m)$.

By pushing the representation theory techniques further, one may try to prove Theorem~\ref{thm:filtration} for more general coefficient fields.

More interestingly, we state the following numerical observation as a conjecture. Let $D_{n,m}^i$ and $E_{n,m}^i$ be defined as in the previous section.
\begin{Conj}\label{conj:surj}
The saddle cobordism $D_{n,n-1}^i\to D_{n,n-1}^{i-1}$ at the crossing of $D_{n,n-1}^i$ corresponding to the last letter $\sigma_i$ induces a surjection in rational Khovanov homology.
\end{Conj}

Equivalently, the conjecture states that the exact triangle
$$\begin{tikzcd}[column sep=5pt]
Kh(E_{n,n-1}^{i-1})[2n-2]\{6n-7\}\ar[rr]&&Kh(D_{n,n-1}^i)\ar[ld]\\
&Kh(D_{n,n-1}^{i-1})\{1\}\ar[lu,"\text{[}1\text{]}"]&
\end{tikzcd}$$
splits into a short exact sequence $$0\to Kh(E_{n,n-1}^{i-1})[2n-2]\{6n-7\}\to Kh(D_{n,n-1}^i)\to Kh(D_{n,n-1}^{i-1})\{1\}\to0.$$

An affirmative answer to Conjecture~\ref{conj:surj} enables one to express the rational Khovanov homology of the torus links $T(n,n)$ entirely in terms of those of the torus knots $T(n'+1,n')$. Explicitly, let $K_n(t,q)$ denote the Poincar\'e polynomial of the Khovanov homology of $T(n+1,n)$ and $L_n(t,q)$ denote that of $T(n,n)$, then Conjecture~\ref{conj:surj} is equivalent to the following.

\begin{Con}\label{conj:Tnn}
The Poincar\'e polynomial of $Kh(T(n,n))$ is recursively defined by $$L_0=1,\ L_1=q^{-1}+q,$$ 
\begin{align*}
L_n=&\,t^{2n-2}(q^{6n-8}+q^{6n-6})L_{n-2}+\sum_{i=1}^{\lfloor\frac{n-1}2\rfloor}C_{i-1}t^{2i(n-i)}q^{6i(n-i)}L_{n-2i}\\
&\,+\sum_{i=0}^{\lfloor\frac{n-2}2\rfloor}\left(\binom{n-2}{i}-\binom{n-2}{i-1}\right)t^{2i(n-i)}q^{6i(n-i)+n-2i-1}K_{n-2i-1},\quad n\ge2.
\end{align*}
Here $C_n=\frac{1}{n+1}\binom{2n}{n}$ is the $n^{\text{th}}$ Catalan number.
\end{Con}

On the other hand, Shumakovitch and Turner have the following conjecture.
\begin{Conj}[{\cite[Conjecture~1.8]{gorsky2013stable}}]\label{conj:Tnn-1}
The Poincar\'e polynomial of $Kh(T(n+1,n))$ is recursively defined by $$K_0=K_1=q^{-1}+q,\ K_2=q+q^3+t^2q^5+t^3q^9,$$ $$K_n=q^{2n-2}K_{n-1}+t^{2n-2}q^{6n-6}K_{n-2}+t^{2n-1}q^{8n-8}K_{n-3},\quad n\ge3.$$
\end{Conj}

Conjecture~\ref{conj:Tnn} and Conjecture~\ref{conj:Tnn-1} together establish a recursive formula for the rational Khovanov homology of $T(n,n)$.

With the help of the computer program KnotJob \cite{knotjob}, we are able to verify Conjecture~\ref{conj:surj} for the cases $n\le8$. The integral version of the conjecture is not true, and the first counterexample is found at $n=7$, $i=6$. We also verified Conjecture~\ref{conj:Tnn-1} for $n\le8$.

One can also consider the (rational) annular Khovanov homology version of Conjecture~\ref{conj:surj}. With the help of the program implemented in \cite{hunt2015computing}, we are able to verify the cases $n\le4$ for the annular version. Due to the existence of a spectral sequence from annular Khovanov homology to the usual Khovanov homology, an affirmative answer to the annular version implies the version as stated. We remark that, however, the annular version of Conjecture~\ref{conj:Tnn-1} is not true.

More generally, if Conjecture~\ref{conj:surj} has an affirmative answer, one can ask for what values of $m,n,i$ does $D_{n,m}^i\to D_{n,m}^{i-1}$ induce a surjection on Khovanov homology, aiming for a complete calculation of rational Khovanov homology of all torus links. This is not true for most cases where $n|m$. For $n\nmid m$ with $m+n\le7$, all violations for the annular version of this question are $(n,m,i)=(4,2,1),(4,2,3),(5,2,3)$.

\printbibliography

\end{document}